\documentclass{article}

\RequirePackage[OT1]{fontenc}
\RequirePackage{verbatim, amscd,epsfig,amsmath,amsthm,amstext,amssymb,tikz}
\usepackage{mathrsfs}
\usepackage[margin=3cm]{geometry}
\usepackage{tikz}
\usetikzlibrary{patterns}
\usepackage{caption}
\usepackage{subcaption}
\usepackage[all]{xy}
\newcommand{\lightercolor}[3]{% Reference Color, Percentage, New Color Name
    \colorlet{#3}{#1!#2!white}
}
\newcommand{\darkercolor}[3]{% Reference Color, Percentage, New Color Name
    \colorlet{#3}{#1!#2!black}
}
\darkercolor{gray}{50}{dgray}
\lightercolor{gray}{50}{lgray}
\tikzset{
l/.style={gray,dashed},
ll/.style={red,dashed},
grn/.style=green!70!black,
b/.style=blue,
r/.style=red,
p/.style=purple}
\usetikzlibrary{plotmarks}
\numberwithin{equation}{section}
\theoremstyle{plain}
\newtheorem*{proposition*}{Proposition}
\newtheorem{theorem}{Theorem}[section]
\newtheorem{lemma}[theorem]{Lemma}
\newtheorem{cor}[theorem]{Corollary}
\newtheorem{corollary}[theorem]{Corollary}

\newtheorem{thm}[theorem]{Theorem}

\theoremstyle{definition}

\newtheorem{defn}[theorem]{Definition}

\newtheorem{remark}[theorem]{Remark}

\renewcommand{\d}{\mbox{dist}}
\newcommand{\M}{\mathcal{M}}
\newcommand{\PHT}{\operatorname{PHT}}
\newcommand{\ECT}{\operatorname{ECT}}
\newcommand{\PHzeroT}{\operatorname{PH_0T}}

\newcommand{\Z}{\mathbb{ Z}}
\newcommand{\coker}{\operatorname{coker}}
\newcommand{\rank}{\operatorname{rank}}
\newcommand{\R}{\mathbb{ R}}

\newcommand{\D}{\mathcal{D}}

\newcommand{\bb}{\mathrm{b}}
\newcommand{\dd}{\mathrm{d}}
\newcommand{\im}{\operatorname{im}}

\usepackage{url}
\title{Persistent Homology Transform for Modeling Shapes and Surfaces}
\author{%
{\sc Katharine Turner},\\[2pt]
Department of Mathematics\\
University of Chicago \\
Chicago, IL 60637 USA\\
{Corresponding author: kate@math.uchicago.edu}\\[6pt]
{\sc Sayan Mukherjee}\\[2pt]
Departments of Statistical Science\\
Computer Science, and Mathematics \\
Duke University \\
Durham, NC 27708 USA\\
{sayan@stat.duke.edu}\\[6pt]
{\sc and}\\[6pt]
{\sc  Doug M Boyer}\\[2pt]
Department of Evolutionary Anthropology \\
Duke University \\
Durham, NC 27708 USA\\
{doug.boyer@duke.edu}}

\begin{document}

%\begin{frontmatter}

\maketitle

%
%\begin{aug}
%\author{\fnms{Katharine} \snm{Turner}\thanksref{m2}\ead[label=e2]{kate@math.uchicago.edu}}
%\author{\fnms{Sayan} \snm{Mukherjee}\thanksref{m1}\ead[label=e1]{sayan@stat.duke.edu}},
%\and
%\author{\fnms{Doug M} \snm{Boyer}\thanksref{m1}\ead[label=e3]{doug.boyer@duke.edu}
%\ead[label=u1,url]{http://www.foo.com}}
%
%\runauthor{K. Turner, S. Mukherjee, D. Boyer}
%
%\affiliation{University of Chicago\thanksmark{m2} and Duke University\thanksmark{m1} }
%
%\address{Katharine Turner\\
%Department of Mathematics\\
%University of Chicago \\
%Chicago, IL 60637 USA\\
%\printead{e2}\\
%\phantom{E-mail:\ }}
%
%
%\address{Sayan Mukherjee\\
%Departments of Statistical Science\\
%Computer Science, and Mathematics \\
%Institute for Genome Sciences \& Policy\\
%Duke University \\
%Durham, NC 27708 USA\\
%\printead{e1}\\
%\phantom{E-mail:\ }}
%
%\address{Doug M Boyer\\
%Department of Evolutionary Anthropology \\
%Duke University \\
%Durham, NC 27708 USA\\
%\printead{e3}
%\phantom{E-mail:\ }}
%\end{aug}
\begin{abstract}
{In this paper we introduce a statistic, the persistent homology transform (PHT), 
to model surfaces in $\R^3$ and shapes in $\R^2$. This statistic is a collection of persistence diagrams
-- multiscale topological summaries used extensively in topological data analysis.
We use the PHT to represent shapes and execute operations such as computing 
distances between shapes or classifying shapes. We prove the map 
from the space of simplicial complexes in $\R^3$ into the
space spanned by this statistic is injective. This implies that the 
statistic is a sufficient statistic for probability densities on the space of 
piecewise linear shapes. We also show that a variant of this statistic, the
Euler Characteristic Transform (ECT), admits a simple exponential
family formulation which is of use in providing likelihood based
inference for shapes and surfaces. We illustrate the utility of 
this statistic on  simulated and real data.}
{persistence homology, surfaces, shape spaces, sufficient shape statistics}\\
Insert classification here
\end{abstract}
%
%\begin{keyword}[class=MSC]
%\kwd[Primary ]{62B99}
%\kwd{14J99}
%\kwd[; secondary ]{14J99}
%\end{keyword}
%
%\begin{keyword}
%\kwd{persistence homology}
%\kwd{surfaces}
%\kwd{shape spaces}
%\kwd{sufficient shape statistics}
%\end{keyword}

%\end{frontmatter}

\section{Introduction}

In this paper we introduce a sufficient statistic, the persistent homology transform (PHT), to model objects and surfaces in $\R^3$ and shapes in $\R^2$. This result is of interest to three communities, the shape statistics community \cite{Bookstein97,DrydenMardia98,Kendall77,Kendall84}, the topological data analysis (TDA) community \cite{BubCarKimLuo2010,Carlsson09,deSilvaGhrist07,EdeHar2010,gamble2010}, and applied statisticians and domain researchers modeling shapes including medical imaging \cite{BandulasiriBP09,Styneretal06} and morphology \cite{Bookstein97,Boyeretal11,Zelditch2004}.

Fitting curves and surfaces to model shapes has many applications in a
variety of fields. A concrete example of central interest to one of
the authors is computing the distance between heel bones in primates
to generate a tree and comparing this tree to a tree generated from
the genetic distances between the primate species 
\cite{Boyeretal13}. The central problem in almost all approaches to modeling surfaces and shapes is obtaining a representation of the shape that can be used in statistical models. 

In this paper we show that a collection of persistence diagrams --
multiscale topological summaries used extensively in topological data
analysis -- are sufficient statistics for shape and surface
models. This is of interest to the topological data analysis community
since it is the first formal demonstration that persistent homology
\cite{Carlsson09,EdeHar2010}, the dominant tool used in TDA,  does not
result in the loss of information. For the shape statistics community
this is the first result that we know of that applies a sufficient statistic for shapes or surfaces (besides the
obvious and not very useful sufficient statistic of the data themselves). Almost all statistical models start with a set of landmarks provided by the user as an 
initialization step, a probability model is then placed on these landmarks. From the perspective of the likelihood principle \cite{BergerWolpert84} statistical inference should proceed from a probability model on the shapes themselves, unless the landmarks are sufficient statistics for shapes. In this paper 
we provide a sufficient statistic for shapes and surfaces. This suggests that in theory a generative or sampling model on shapes or surfaces should
 be possible in shape statistics. Indeed in Section \ref{expfamily} we
 show how we can use sufficient statistics for likelihood based inference.
 
Statistical models of shapes (characterized as a set of landmarks)
were pioneered in the works of Kendall and Bookstein \cite{Bookstein97,Kendall77,Kendall84}. The central idea developed in this line of work was the shape space, a differentiable manifold often with appropriate Riemannian structures. See \cite{BandulasiriBP09, Bhattacharya2008,Bhattacharya2005,Dryden2013,Schmidler07}
for recent results on statistical analysis of shapes. Another line of
research we draw from is modeling shapes using multiscale topological
summaries of data. The key idea we draw upon from this discipline is
the elevation function \cite{Agarwaletal04,Wang05} which was developed
as an application of discrete Morse theory to problems in protein
structure modeling. 

An alternative approach to modeling shapes comes from the formulation
by Grenander \cite{DepuisGrenander98}. In this formulation shapes are
considered as points on an infinite-dimensional manifold and variation
in shape is modeled by the action of Lie groups on these
manifolds. This is a very appealing paradigm but is computationally
intensive and requires the parameterization of the shape manifold.

The ideas we present in this paper are closely related to ideas in
integral geometry that have been used to model surfaces and random
fields
\cite{stoyan_stochastic_1987,worsley_boundary_1995,worsley_estimating_1995}
as well as point processes \cite{Diggle2003, MW2003,Ripley,stoyan_stochastic_1987}. The central idea to the integral geometric
approach was to study invariant integral transforms from the space of
functions on surfaces or shapes to spaces of functions that are more
convient for analysis such as functions on an interval of the real line.
The idea is that one can more easily
manipulate, compute, and model in the transformed space. A classic
example of a widely used integral transform is the Radon 
transform \cite{Schapira91}, see \cite{klain1997introduction} for
details on classic ideas in integral geometry including Minkowski
functionals and Hadwiger integrals.

We begin the paper with topological preliminaries and relevant definitions in Section \ref{prelims}.  In Section \ref{section:suff}
we first state and prove conditions under which the PHT is a sufficient statistic for surfaces and shapes. We then state sufficiency results
for the  $0$-th dimensional PHT for surfaces that are homeomorphic to a sphere. We end the section with a discussion on the setting when the objects
are not aligned and we have to quotient out rotation, scaling, and rotation as is normally done in shape statistics. In Section 
\ref{results} we show the efficacy of our method for computing distances between shapes and surfaces in simulated data as well as real data.
For the simulated data we demonstrate that we can work with unaligned
object. We close with a discussion.

\subsection{A motivating example}

The classical problem in morphology of measuring distances between bones often
is realized as measuring the distance between the surfaces of the
bones. Historically this problem has been very amenable to classical
shape statistics as the information about a bone was stored as a set
of landmark points on the bone and distances between the landmark 
points. However, with the increased prevalence of 
scanning technologies such as computerized tomography (CT) scans
bones are now often represented as meshes. In Figure \ref{teeth}
we display a snapshot of the meshes of five teeth. Understanding
variation in a set of bones or teeth by providing estimates of distances
between the surfaces in an automated fashion 
\cite{Boyeretal11,Boyeretal13,Gladmanetal13} would be of great
practical importance.

{\begin{figure}[hbt]
\begin{center}
\includegraphics[height=1.4in]{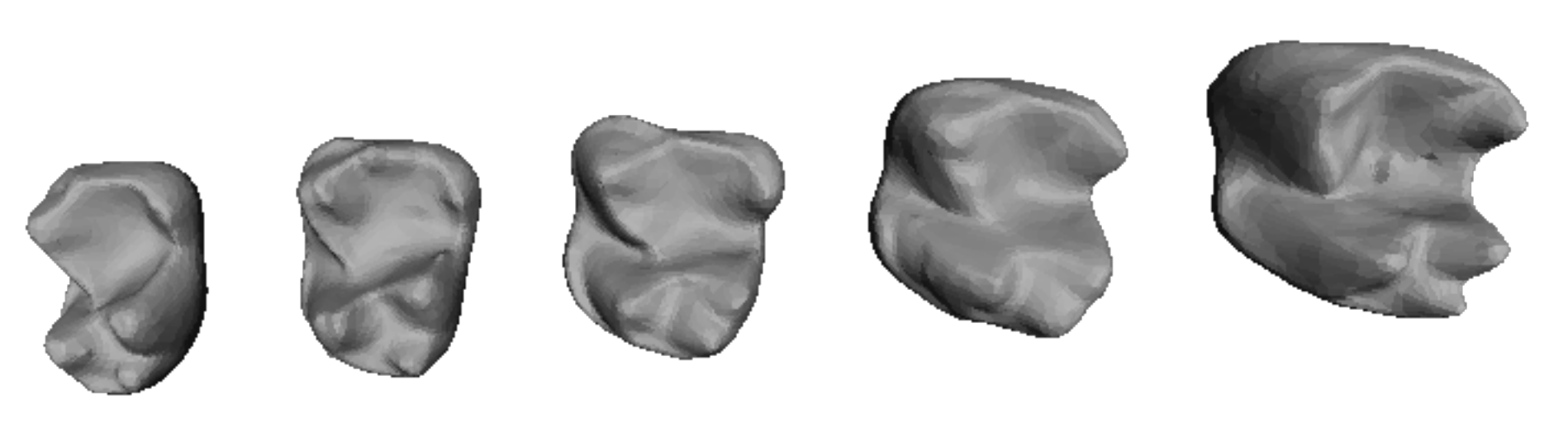}
\caption{\label{teeth} Images of the meshes of five teeth. A common
  problem in morphology is to measure distances between these five teeth.
}
 \end{center}
\end{figure}

In \cite{Boyeretal11} a procedure to measure distances between surfaces,
such as the boundary surfaces of teeth, was developed based on using conformal geometry to
construct flattened representations of pairs of surfaces followed by
continuous Procrustes distance \cite{LipmanDaub13} to measure the
distance between the surfaces. A setting when this approach will have
problems is if one wants to measure distances between objects that are
not isomorphic. For example, if one of the teeth were broken then
generating a conformal map from the broken region of the tooth to the
corresponding intact region of the other tooth will be a problem. 

We will apply the PHT to measure distances between
surfaces. Specifically, in Section \ref{results} we will use the 
PHT to measure distances beween the heel bones of 106 primates. The basic
approach will be to transform each bone using the PHT and use standard
distances between persistencce diagrams to measure pairwise distances between
bones. We will then cluster these distances to propose evolutionary
relations between the primate species. One ultity of our approach is 
that we do not have to compute a correspondence between the bones, 
as is required in a conformal map. In the case where correspondences
are very unstable as would be the case when objects are 
not isomorphic our procedure should be more robust.

\section{Persistence diagrams and height functions}\label{prelims}

\subsection{Definitions and topological preliminaries}

Persistent homology is a computational method for measuring changes in homology of a filtration of simplicial complexes. We first review the notion of a simplicial complex and simplicial homology. The computation of persistent homology requires a field, in general simplicial homology can be computed over any ring. In this paper and in most of topological data analysis the field is $\Z_2$, due to computational reasons.

Simplices are the elementary objects on which we will operate. Examples are points, lines, triangles, and $k$-dimensional generalizations. Formally,
a \emph{$k$-simplex} is the convex hull of $k+1$ affinely independent points $v_0,v_1, \ldots v_k$ and is denoted $[v_0,v_1,\ldots,v_k]$. For example, the $0$-simplex $[v_0]$ is the vertex $v_0$, the $1$-simplex $[v_0,v_1]$ is the edge between the vertices $v_0$ and $v_1$, and the $2$ simplex $[v_0, v_1, v_2]$ is the triangle bordered by the edges $[v_0,v_1]$, $[v_1, v_2]$ and $[v_0, v_2]$. 
%Simplices are in general orientedThere is an orientation on simplices. If $\tau$ is a permutation then $[v_0,v_1,\ldots,v_k] = (-1)^{\operatorname{sgn}(\tau)}[v_{\tau(0)}, v_{\tau(1)}, \ldots , v_{\tau(k)}]$. However, if we are considering homology over $\Z_2$ which is the norm in applications, $1=-1$ and we can ignore orientation.

A simplicial complex consists of simplices glued together with certain rules. To define the rules we first define the face of a simplex. We call $[u_0, u_1, \ldots u_j]$ a \emph{face} of $[v_0,v_1, \ldots v_k]$ if $\{u_0, u_1, \ldots u_j\}\subset \{v_0,v_1, \ldots v_k\}$. A \emph{simplicial complex} $M$ is a countable set of simplices such that
\begin{enumerate}
\item every face of a simplex in $M$ is also in $M$;
\item if two simplices $\sigma_1,\sigma_2$ are in $M$ then their intersection is either empty or a face of both $\sigma_1$ and $\sigma_2$.
\end{enumerate}

Given finite simplicial complex $K$, a \emph{simplicial $k$-chain} is a formal linear combination (over $\Z_2$ in this paper) of $k$-simplices in $K$. The set of $k$-chains forms a vector space $C_k(K)$. We define the boundary map $\partial_k:C_k(K) \to C_{k-1}(K)$ as
$$\partial_k\left([v_0, v_1, \ldots v_k]\right) = \sum_{j=0}^k (-1)^j[v_0,\ldots, v_{j-1}, v_{j+1},  \ldots v_k]$$
and extending linearly. 

Elements of $B_k(K) = \operatorname{im} \partial_{k+1}$ are called boundaries and elements of $Z_k(K)=\operatorname{ker} \partial_k$ are called cycles.\footnote{$\operatorname{im}(f)= \{y: f(x)=y \text{ for some $x$}\}$ is called the image of $f$ and $\operatorname{ker}(f) = \{x: f(x)=0\}$ is called the kernel of $f$.}  Direct computation shows $\partial_{k+1}\circ\partial_k=0$ and hence $B_k(K) \subseteq Z_k(K)$. This allows us to define the $k$-th \emph{homology group} of $M$ as
$$H_k(K):=Z_k(K)/B_k(K).$$

We now consider the construction of persistence diagrams. We are given a filtration $K = \{K_r | r\in \R\}$ of a countable simplicial complex indexed over the positive real numbers, thought of as time. By this we mean that each $K_a$ is a simplicial complex and that $K_a \subseteq K_b$ for $a<b$. We wish to summarize how the topology of the filtration changes over time. For $a <b$ we have an inclusion map of simplicial complexes $\iota: K_a \to K_b$. This induces inclusion maps $$\iota:B_k(K_a) \to B_k(K_b) \quad \text{and}\quad \iota: Z_k(K_a) \to Z_k(K_b).$$
This induces homomorphisms (which are generally not inclusions) $$\iota_k^{a\to b}:H_k(K_a) \to H_k(K_b).$$
We can define the persistence homology groups $H_k(a,b)$ by
$$H_k(a,b) :=Z_k(K_a)/(Z_k(K_a) \cap B_k(K_b)).$$
$H_k(a,b)$ is the group of homology classes in $H_k(K_a)$ which persist to $H_k(K_b)$ or in other words the image of $\iota_k^{a\to b}$.

We say that a homology class $\alpha \in H_k(K_i)$ is \emph{born} at time $a$ (denoted $\bb(\alpha)$) if it is in the cokernel of $\iota_k^{a' \to a}$ for any $a'<a$.The cokernel of $f:X\to Y$ is $Y/\operatorname{im}f$. It can be thought of as the vector subspace  of $Y$ which is perpendicular to $\operatorname{im}f$. More precisely we can say an entire coset is born but this can be represented by the element in the vector space perpendicular to $\im f$.

 For $\alpha$ born at time $a$, we say that $\alpha$ \emph{dies} at time $b$ (denoted $\dd(\alpha)$) if for all $a'<a<b'<b$ we have $\iota_k^{a\to b}(\alpha) \in\im  \iota_k^{a' \to b}$ but $\iota_k^{a\to b'}(\alpha) \notin \operatorname{im}  \iota_k^{a' \to b' }$. Informally we can think of the process of dying as either becoming zero or merging into a pre-existing homology class. For example, suppose we have two connected components, one represented by the $H_0$ class $\alpha_0$ and is born at time $0$ and the other represented by the $H_0$ class $\alpha_1$ and is born at time $1$. If these components become connected at time $2$, then we say $\alpha_1$ dies at time $2$. We say that $\alpha$ is an \emph{essential class} of $K$ if it never dies. We say the homology class $\alpha$ has \emph{persistence} $\dd(\alpha)-\bb(\alpha)$.
%
%For each pair $(i,j)$ with $i<j$ we can then consider the vector space . Let $\beta_k^{(i,j)}$ denote the. Similarly let $\beta_k^{(i,\infty)}$ denote the dimension of the space of essential $k$-dimensional homology classes that are born at time $i$.
%% and let $\beta_*^{(-\infty,j)}$ denote the dimension of the space of homology classes that die at time $j$ and were never born.

Let $\R^{2+}:=\{(a,b) \in (-\infty \cup \R) \times (\R\cup \infty): a<b\}$. We define the $k$-th persistence diagram corresponding to the filtration $K$ to be the multi-set of points in $\R^{2+}$ alongside countably infinite copies of the diagonal such that the number of points (counting multiplicity) in $[-\infty, a] \times [b,\infty]$ is equal to the dimension of $H_k(a,b)$. That is, it is equal to the dimension of the space of $k$-dimensional homology classes that are born at or before $a$ and die at or after $b$. This is achieved by placing at each $(a,b)$ a number of points equal to dimension of the space of $k$-dimensional homology classes that are born at time $a$ and die at time $b$. The countably infinite copies of the diagonal play the role of persistent homology classes whose persistence is zero and hence would not otherwise seen. 

We restrict our attention to persistence diagrams such that 
$$\sum_{ \alpha \text{ not essential}} \dd(\alpha)-\bb(\alpha) <\infty.$$
This is automatically true if the persistence diagrams contain finitely many off diagonal points. 

Let us consider an example. Consider the simplicial complex in the plane shown in Figure 2 and let us use a filtration by sublevel sets of vertical height (as shown by the arrow).

\begin{figure}[hbt]

\begin{center}

\begin{tikzpicture}[scale=0.75 , baseline=5pt]\label{fig:ex_plane}
\fill (0,0) node [below] {$v_0$}circle (2pt);
\fill (-1, 0.5) node [below] {$v_1$}circle (2pt);
\fill (1, 0.75) node [right] {$v_2$}circle (2pt);
\fill (0.5, 1.5) node [right] {$v_3$}circle (2pt);
\fill (-0.5, 2) node [above] {$v_5$}circle (2pt);
\fill (-2, 1.25) node [above] {$v_4$}circle (2pt);
\draw (-1,0.5)--(-0.5,2)--(0.5,1.5)--(1,.75)--(0,0)--(0.5,1.5);
\draw[pattern=north west lines, pattern color=black] (0,0)--(0.5,1.5)--(1,0.75);
\draw (2.75,1.75)--(3,2)--(3.25,1.75);
\draw (3,0)--(3,2);
\end{tikzpicture} 
\end{center}
\caption{The simplicial complex $K$  and the vertical height direction}
\end{figure}
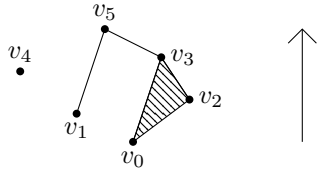

The subcomplex at time $t$ is the set of all simplices that entirely lie at or below height $t$. This constructs the filtration in Figure 3. We then can keep track of how the $0$-th dimensional homology changes as we progress through the filtration. 
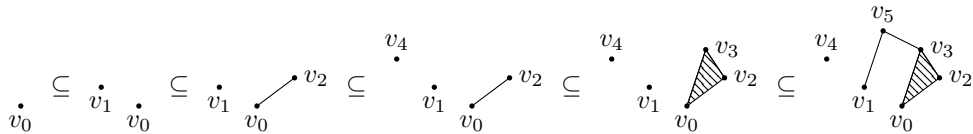
\begin{figure}[hbt]\label{fig:ex_filt_plane}
\begin{center}
\begin{tikzpicture}[scale=0.5 , baseline=5pt]
\fill (0,0) node [below] {$v_0$}circle (2pt);
\end{tikzpicture} 
$\subseteq$
\begin{tikzpicture}[scale=0.5 , baseline=5pt]
\fill (0,0) node [below] {$v_0$}circle (2pt);
\fill (-1, 0.5) node [below] {$v_1$}circle (2pt);
\end{tikzpicture} 
$\subseteq$
\begin{tikzpicture}[scale=0.5 , baseline=5pt]
\fill (0,0) node [below] {$v_0$}circle (2pt);
\fill (-1, 0.5) node [below] {$v_1$}circle (2pt);
\fill (1, 0.75) node [right] {$v_2$}circle (2pt);
\draw(0,0)--(1,0.75);
\end{tikzpicture} 
$\subseteq$
\begin{tikzpicture}[scale=0.5 , baseline=5pt]
\fill (0,0) node [below] {$v_0$}circle (2pt);
\fill (-1, 0.5) node [below] {$v_1$}circle (2pt);
\fill (1, 0.75) node [right] {$v_2$}circle (2pt);
\fill (-2, 1.25) node [above] {$v_4$}circle (2pt);
\draw(0,0)--(1,0.75);
\end{tikzpicture} 
$\subseteq$
\begin{tikzpicture}[scale=0.5 , baseline=5pt]
\fill (0,0) node [below] {$v_0$}circle (2pt);
\fill (-1, 0.5) node [below] {$v_1$}circle (2pt);
\fill (1, 0.75) node [right] {$v_2$}circle (2pt);
\fill (0.5, 1.5) node [right] {$v_3$}circle (2pt);
\fill (-2, 1.25) node [above] {$v_4$}circle (2pt);
\draw (0.5,1.5)--(1,.75)--(0,0)--(0.5,1.5);
\draw[pattern=north west lines, pattern color=black] (0,0)--(0.5,1.5)--(1,0.75);
\end{tikzpicture} 
$\subseteq$
\begin{tikzpicture}[scale=0.5 , baseline=5pt]
\fill (0,0) node [below] {$v_0$}circle (2pt);
\fill (-1, 0.5) node [below] {$v_1$}circle (2pt);
\fill (1, 0.75) node [right] {$v_2$}circle (2pt);
\fill (0.5, 1.5) node [right] {$v_3$}circle (2pt);
\fill (-0.5, 2) node [above] {$v_5$}circle (2pt);
\fill (-2, 1.25) node [above] {$v_4$}circle (2pt);
\draw (-1,0.5)--(-0.5,2)--(0.5,1.5)--(1,.75)--(0,0)--(0.5,1.5);
\draw[pattern=north west lines, pattern color=black] (0,0)--(0.5,1.5)--(1,0.75);
\end{tikzpicture} 

\end{center}
\caption{The filtration of $K$ by height in direction $v$. Each simplex is included at its maximal height}
\end{figure}

A component ($v_0$) is born in the first stage of the filtration at time $t_0$. This component corresponds to an essential class that lives throughout the the rest of the filtration.  It corresponds to a point in the persistence diagram at $(t_0,\infty)$. At $t_1$ of the filtration a new component appears ($v_1$). It joins the first component at the last stage so corresponds to a point in the persistence diagram at $(t_1, t_5)$. Another component appears ($v_4$) at $t_3$. This component is always separate and hence it corresponds to an essential class. It is represented by a point in the persistence diagram at $(t_3, \infty)$. In this example there is no homology class of dimensions greater than $0$ so the higher dimensional persistence diagrams have no off diagonal points.

Let $\D$ denote the space of persistence diagrams. There are many choices of metric on $\D$ just like there are many choices of metric on spaces of functions. It is worth mentioning that the coordinates of the points in the persistence diagrams have special meanings and hence deserve to be treated (somewhat) individually. A small, localised change in a filtration will often affect only one the two coordinates. Let $X$ and $Y$ be persistence diagrams. We can considers bijections $\phi$ between the points and copies of the diagonal in $X$ and the points and copies of the diagonal in $Y$. These bijections are the transport plans that we consider. Bijections always exist because there are countably many copies of the diagonal which everything can be paired with. Define
\begin{align}\label{eq:dp}
{\d}_p(X,Y) &=\left(\inf_{\phi:X \to Y} \sum_{x\in X} \|x-\phi(x)\|_p^p\right)^{1/p}.
\end{align}  
We call a bijection optimal if it achieves the infimum in \eqref{eq:dp}. That such an optimal bijection always exists (but is not necessarily unique) is proved in \cite{Turner13}. We illustrate in Figure \ref{fig:optimalbijection} an example of the optimal bijection between two persistence diagrams. To see more details about the range of metric choices see 
 \cite{Turner13,villani2009optimal}.
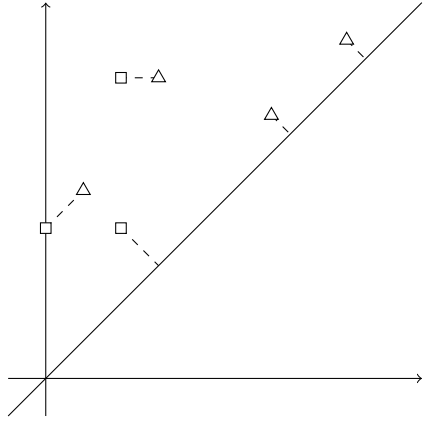
\begin{figure}[hbt]

\begin{center}
\begin{tikzpicture}[scale=.5]
\draw []   (-1,-1) -- (10,10);
\draw [->] (-1, 0) -- (10,0);
\draw [->] ( 0,-1) -- (0,10);

\draw[dashed] (2,4) -- (3,3);
\draw[dashed] (0,4) -- (1,5);
\draw[dashed] (3,8) -- (2,8);
\draw[dashed] (6,7) -- (6.5,6.5);
\draw[dashed] (8,9) -- (8.5,8.5);

\draw plot[only marks, mark=square*,mark options={ mark size=4pt, fill=white}] coordinates{
(0,4) (2,4)(2,8)};

\draw plot[only marks, mark=triangle*,mark options={mark size=6pt, fill=white}] coordinates{
(1,5)(3,8) (6,7)(8,9)};

\end{tikzpicture}
\end{center}
\caption{The dashed lines indicate an optimal bijection from the persistence diagram consisting of the square points (and copies of the diagonal) to the persistence diagrams consisting of the triangle points (and copies of the diagonal).}\label{fig:optimalbijection}
\end{figure}
We will consider a choice of metric which is analogous to $1$-Wasserstein distances on the space of measures, or $L_1$ distances on the space of functions on a discrete set, with $p=1$ in \eqref{eq:dp}. From now on let $\d(X,Y)$ denote $\d_1(X,Y)$. The sufficiency results automatically hold for any choice of metric but we found the $1$-Wasserstein distance to perform better than other metrics in empirical studies. We suspect that the variation in performance for different distance metrics is driven by variation in the pairing of points to the diagonal.%The regions in the plane where such pairings occur is illustrated in \cite{Medians}.

\subsection{Computing the persistence homology transform}

Let $M$ be a subset of $\R^d$ which can be written as a finite simplicial complex. For any unit vector $v \in S^{d-1}$ we define a filtration $M(v)$ of $M$ parameterized by a height $r$ where
$$M(v)_r = \{x \in M: x\cdot v \leq r\}$$
is the subcomplex of $M$ containing all the simplices below height $r$ in the direction $v$. $M(v)_r$ and $ \{ \Delta\in M: x\cdot v \leq r \text{ for all } x\in \Delta\}$ are homotopy equivalent\footnote{Two spaces are ``homotopy eqivalent'' if we can continuously deform (without tearing or glueing) one into the other. Homology is invariant under such continuous changes.} and hence their homologies are the same. The $k$-th dimensional persistence diagram, $X_k(M,v)$, summarizes how the topology of the filtration $M(v)$ changes over the height parameter $r$. By stability results on persistence diagrams \cite{CohEdeHar2007,Turneretal13},  the map $v \mapsto X_k(M,v)$ is continuous.

\begin{lemma}\label{lem:cont}
The map $v \mapsto X_k(M,v)$ is Lipschitz (and hence also continuous) with respect to the distance metric $\d(\cdot,\cdot)$ whenever $M$ is a finite simplicial complex.
\end{lemma}
\begin{proof}
Since $M$ is a finite simplicial complex there is a bound $N$ on the number of off diagonal points in any diagram $X_k(M,v)$.  There also is a bound $K$ on the distance of any point in $M$ to the origin. Consider the functions $h_{v_1}$ and $h_{v_2}$ on $M$ which are the height functions in directions $v_1$ and $v_2$ respectively. That is for $x$ in $M$ we have $h_{v_i}(x) =x \cdot v_i$. Now $$|h_{v_1}(x) - h_{v_2}(x)|=|x \cdot v_1-x \cdot v_2|\leq \|x\|_2 \|v_1-v_2\|_2\leq K\|v_1-v_2\|_2.$$
and hence
\begin{align}\label{eq:bottle}
\|h_{v_1} - h_{v_2}\|_\infty \leq K\|v_1-v_2\|_2.
\end{align}
The bottleneck stability theorem tells us that
$$\d(X_k(M,v_1),X_k(M,v_2))\leq N \|h_{v_1} - h_{v_2}\|_\infty.$$ Combined with 
\eqref{eq:bottle} we can conclude
 $$\d(X_k(M,v_1),X_k(M,v_2)) \leq NK \|v_1-v_2\|_2$$ and hence $v \mapsto X_k(M,v)$ is Lipshitz with respect to $\d_1$.
\end{proof}

The above lemma generalizes to all $p\geq 1$ for the family of distance metrics  \eqref{eq:dp} on $\D$. 
\begin{remark}\label{rem:p=infty}
The above lemma generalizes to all $p\geq 1$ for the family of distance metrics  \eqref{eq:dp} on $\D$. In particular, in the case $p=\infty$ the Lipschitz constant is bounded by the distance of the furtherest point in $M$ to the origin.
\end{remark}

\begin{defn}
The \emph{persistent homology transform} of $M\subset \R^d$ is the function
\begin{align*}
 \PHT(M): S^{d-1} &\to \D^{d}\\
 v&\mapsto (X_0(M,v), X_1(M,v), \ldots, X_{d-1}(M,v)).
 \end{align*}
\end{defn}

By Lemma \ref{lem:cont} we know the PHT of a finite simplicial complex is continuous. Let $C(X,Y)$ denote the space of continuous functions from $X$ to $Y$. We have shown that for $M \subset \R^d$, if $M$ has a finite simplicial complex representation, then $\PHT(M) \in C(S^d, \D^d)$.

Let $\mathcal{M}_d$ be the space of subsets of $\R^d$ that can be written as finite simplicial complexes. More precisely we can think of pairs $(K, f)$ where $K$ is a finite simplicial complex, and $f:K \to \R^d$ such that the restriction of $f$ to any simplex in $K$ is linear and the preimage under $f$ of every point in the image of $K$ is starlike. Observe that this last condition ensures that $f(K)$ is homotopy equivalent to $K$. We then define $\mathcal{M}_d$ to be the space of all pairs $(K,f)$ under the equivalence $(K_1,f_1) \sim (K_2, f_2)$ when $f_1(K_1)=f_2(K_2)$.

The main result of this paper is that the PHT is a sufficient statistic. If the PHT is injective it will be a sufficient statistic. For $\M_d$ with $d=2,3$ the transform is injective. The following fact is a composition of two of the main theoretical results in this paper, Theorem \ref{inject} and Corollary \ref{cor:inject}. 
\begin{proposition*}
The persistent homology transform is injective when the domain is $\M_d$ for $d=2,3$. 
\end{proposition*}

Corollary \ref{sufdist} uses the above fact to prove that the PHT is a sufficient statistic for distributions on $\M_d$.
We provide in Section \ref{section:suff} a proof of the above statement. The proof is constructive and provides an algorithm that  reconstructs a
simplicial complex from the persistence diagrams that compose the PHT of the simplicial complex. Thus the proof also shows that the persistent homology transform is theoretically invertible. The PHT can be used to define a distance metric between shapes or surfaces 
\begin{equation}
\label{distance}
\d_{\mathcal{M}_d}(M_1, M_2) :=\sum_{k=0}^d \int_{S^{d-1}} \d(X_k(M_1,v),X_k(M_2, v)) dv.
\end{equation}
We can show that $\d_{\mathcal{M}_d}$ is a distance metric on $\mathcal{M}_d$.

Simplicial complexes that are homeomorphic to a sphere are a class of $\mathcal{M}_d$ of independent interest. For this class the $0$-th dimensional persistence diagrams are sufficient to characterize the simplicial complexes. The advantage of this is that the computation of the $0$-th dimensional homology persistence diagrams 
is very fast. One can use a union-find algorithm which is (almost)
linear in the number of vertices in the the simplical complex. The
following fact is discussed further in Section \ref{homeomorphic}.
\begin{proposition*}\label{prop:homeomorphicsphere}
Given a  simplicial complex $M\subset \R^3$ (respectively $\R^2$) which is homeomorphic to $S^2$ or $S^1$ (respectively $S^1$) then one can construct $X_k(M,v)$ from $X_0(M,v)$ and $X_0(M,-v)$ for $k=1,2$.
\end{proposition*}

The above proposition motivates the following definition of the $0$-th dimensional PHT.
\begin{defn}
The \emph{$0$-th dimensional persistent homology transform} of $M \in \R^d$ is the function 
\begin{align*}
 \PHzeroT(M): S^{d-1} &\to \D\\
 v&\mapsto X_0(M,v).
 \end{align*}
\end{defn}

We now will illustrate an example of the $0$-th dimensional persistent homology transform of a simplicial complex shaped like the letter $M$ in the plane. We have the persistence diagrams generated by height functions in 8 different directions illustrated in Figure  \ref{fig:letterM}. This is a discretization of the persistent homology transform of that particular embedding of the letter $M$ in the plane.
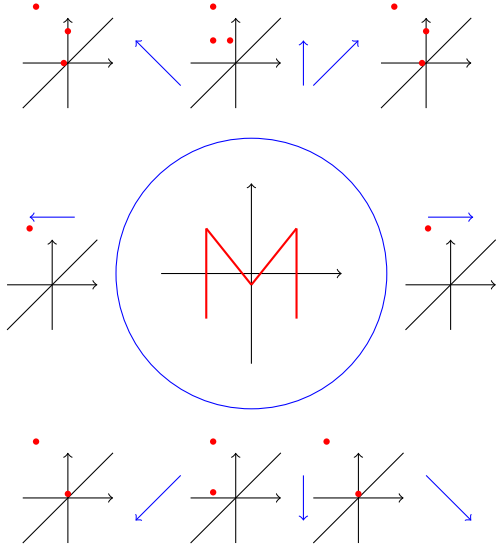
\begin{figure}

\begin{center}
\begin{tikzpicture}[scale=.3]
\draw []   (-2,-2) -- (2,2);
\draw [->] (-2, 0) -- (2,0);
\draw [->] ( 0,-2) -- (0,2);
\draw[b] [->] (5,-1) -- (3,1);
%
%\draw[l] (-1,12) -- (2,9);
%\draw[l] (1,8) -- (2,8);
%\draw[ll] (1,9) -- (1,8);
%\draw[ll] (2,9) -- (2,8);

\fill[r] (-1.41,2.5) circle (4pt)
(-0.18, 0) circle (4pt)
(0,1.41) circle (4pt);
%\fill[b] (2,9) circle (4pt)
%(1,8) circle (4pt);

\end{tikzpicture}
\begin{tikzpicture}[scale=.3]
\draw []   (-2,-2) -- (2,2);
\draw [->] (-2, 0) -- (2,0);
\draw [->] ( 0,-2) -- (0,2);
\draw[b] [->] (3,-1) -- (3,1);
%
%\draw[l] (-1,12) -- (2,9);
%\draw[l] (1,8) -- (2,8);
%\draw[ll] (1,9) -- (1,8);
%\draw[ll] (2,9) -- (2,8);

\fill[r] (-1,2.5) circle (4pt) 
(-0.25, 1) circle (4pt)
(-1,1) circle (4pt);

%\fill[b] (2,9) circle (4pt)
%(1,8) circle (4pt);

\end{tikzpicture}
\begin{tikzpicture}[scale=.3]
\draw []   (-2,-2) -- (2,2);
\draw [->] (-2, 0) -- (2,0);
\draw [->] ( 0,-2) -- (0,2);
\draw[b] [->] (-5,-1) -- (-3,1);
%
%\draw[l] (-1,12) -- (2,9);
%\draw[l] (1,8) -- (2,8);
%\draw[ll] (1,9) -- (1,8);
%\draw[ll] (2,9) -- (2,8);

\fill[r] (-1.41,2.5) circle (4pt)
(-0.18,0) circle (4pt)
(0,1.41) circle (4pt);
%\fill[b] (2,9) circle (4pt)
%(1,8) circle (4pt);

\end{tikzpicture}
\end{center}

\begin{center}

\raisebox{-0.5\height}{
\begin{tikzpicture}[scale=.3]
\draw []   (-2,-2) -- (2,2);
\draw [->] (-2, 0) -- (2,0);
\draw [->] ( 0,-2) -- (0,2);
\draw[b] [->] (1,3) -- (-1,3);
%
%\draw[l] (-1,12) -- (2,9);
%\draw[l] (1,8) -- (2,8);
%\draw[ll] (1,9) -- (1,8);
%\draw[ll] (2,9) -- (2,8);

\fill[r] (-1,2.5) circle (4pt);
%\fill[b] (2,9) circle (4pt)
%(1,8) circle (4pt);

\end{tikzpicture}}
\raisebox{-0.5\height}{
\begin{tikzpicture}[scale=.6]

\draw [->] (-2, 0) -- (2,0);
\draw [->] ( 0,-2) -- (0,2);
\draw[b] (0,0) circle(3);
%\fill[r](-1,-1) circle (2pt)
%(1,1) circle (2pt)
%(-1,1) circle (2pt)
%(1,-1) circle (2pt)
%(0,-0.25) circle (2pt);
\draw[thick,r] (-1,-1) -- (-1,1);
\draw[thick, r] (1,-1) -- (1,1);
\draw[thick, r](1,1)--(0,-.25);
\draw[thick, r](0,-.25)--(-1,1);

%
%\draw[l] (-1,12) -- (2,9);
%\draw[l] (1,8) -- (2,8);
%\draw[ll] (1,9) -- (1,8);
%\draw[ll] (2,9) -- (2,8);
\end{tikzpicture}}
\raisebox{-0.5\height}{
\begin{tikzpicture}[scale=.3]
\draw []   (-2,-2) -- (2,2);
\draw [->] (-2, 0) -- (2,0);
\draw [->] ( 0,-2) -- (0,2);
\draw[b] [->] (-1,3) -- (1,3);
%\draw[l] (1,9) -- (2,9);
%\draw[l] (1,8) -- (2,8);
%\draw[ll] (1,9) -- (1,8);
%\draw[ll] (2,9) -- (2,8);

\fill[r] (-1,2.5) circle (4pt);
%\fill[b] (2,9) circle (4pt)
%(1,8) circle (4pt);

\end{tikzpicture}}
\end{center}

\begin{center}

\begin{tikzpicture}[scale=.3]
\draw []   (-2,-2) -- (2,2);
\draw [->] (-2, 0) -- (2,0);
\draw [->] ( 0,-2) -- (0,2);
\draw[b] [->] (5,1) -- (3,-1);
%
%\draw[l] (-1,12) -- (2,9);
%\draw[l] (1,8) -- (2,8);
%\draw[ll] (1,9) -- (1,8);
%\draw[ll] (2,9) -- (2,8);

\fill[r] (-1.41,2.5) circle (4pt)
(0,0.18)circle(4pt);
%\fill[b] (2,9) circle (4pt)
%(1,8) circle (4pt);
\end{tikzpicture}
\begin{tikzpicture}[scale=.3]
\draw []   (-2,-2) -- (2,2);
\draw [->] (-2, 0) -- (2,0);
\draw [->] ( 0,-2) -- (0,2);
\draw[b] [->] (3,1) -- (3,-1);
%
%\draw[l] (-1,12) -- (2,9);
%\draw[l] (1,8) -- (2,8);
%\draw[ll] (1,9) -- (1,8);
%\draw[ll] (2,9) -- (2,8);

\fill[r] (-1,2.5) circle (4pt)
(-1,0.25) circle (4pt);
%\fill[b] (2,9) circle (4pt)
%(1,8) circle (4pt);
\end{tikzpicture}
\begin{tikzpicture}[scale=.3]
\draw []   (-2,-2) -- (2,2);
\draw [->] (-2, 0) -- (2,0);
\draw [->] ( 0,-2) -- (0,2);
\draw[b] [->] (3,1) -- (5,-1);
%
%\draw[l] (-1,12) -- (2,9);
%\draw[l] (1,8) -- (2,8);
%\draw[ll] (1,9) -- (1,8);
%\draw[ll] (2,9) -- (2,8);

\fill[r] (-1.41,2.5) circle (4pt)
(0,0.18) circle (4pt);
%\fill[b] (2,9) circle (4pt)
%(1,8) circle (4pt);
\end{tikzpicture}
\end{center}
\caption{The $0$-th dimensional persistence diagrams corresponding to filtrations of the letter $M$ by height functions of 8 different directions. The points with $\infty$ in the second coordinate are represented by a point with second coordinate just above the axes.}\label{fig:letterM}

\end{figure}

We can define distance metrics for simplicial complexes homeomorphic to the sphere. Define $\mathcal{S}(d,k)$ as the space of simplicial complexes in 
$\R^d$ homeomorphic to $S^k$. Proposition \ref{prop:homeomorphicsphere} suggests the following alternative metric for the spaces $\mathcal{S}(2,1),\mathcal{S}(3,1)$ and $\mathcal{S}(3,2)$.
\begin{align*}
\d_{\mathcal{S}(d,k)}(M_1, M_2) := \int_{S^{d-1}} \d(X_0(M_1,v), X_0(M_2, v))\,dv
\end{align*}

\section{Injectivity of the transform}\label{section:suff}

We first prove that the map from a space of well-behaved shapes
into the space of PHTs is an injective map. This injective property
will imply that the PHT is a sufficient statistic, it will also imply 
a method for the alignment of shapes. In the final subsection we 
discuss a slight variant of the PHT which we call the Euler Characteristic 
Transform (ECT) for which we can define exponential family models on
shapes and surfaces.

\subsection{Injectivity}

%Let $\M_n$ be the space of a piecewise linear simplicial complexes lying in $\R^n$ with finitely many vertices. For each unit vector $v\in S^{n-1}$ we can consider the height function $h_v:x \to v\cdot x$ in the direction of $v$ on the simplicial complex $M$. The persistence homology transform for elements in the
%space $\M_n$ is function over $S^{n-1}$ to sets of persistence diagram, one for each of the dimension. Given a vector $v\in S^{n-1}$ we $\PHT(M)(v)$  denoted $X(m,v)$ constructed form the sublevels sets of $h_v$. 
%Define the function $F:\M\to C(S^{n-1},\D$ by $\M \mapsto \{(v,f_v(\M)):v\in S^2\}$ as a map from the shape space into the PHT space.
%The homology classes can be defined over any choice of field.

\begin{thm}\label{inject}
The persistent homology transform is injective when the domain is $\M_3$.
\end{thm}

\begin{proof}
The proof is constructive. We state the proof as an algorithm. Given a function $\PHT(M):S^2 \to \D^3$ we state a procedure to find all the vertices in one of the simplest
representation of the simplicial complex, by simplest we mean one with the fewest possible number of vertices. We then determine the link of each vertex.
Since $M$ is assumed to be piecewise linear computing the vertices and links is enough for reconstruction.

We first provide some facts that we will use in the procedures to reconstruct $M$ from $\PHT(M)$: 
\begin{enumerate}
\item[(1)] Changes in homology of sublevel sets of height functions in any direction can only occur at the heights of vertices of $M$. 
\item[(2)] Every vertex $x$ determines a critical point for an open ball in the set of all directions (recall that the set of all directions is $S^2$). That is to say that the inclusion of this point $x$ causes a birth or a death of a homology class. The homology class is also consistent inside the ball.  We claim that there is some ball of $\{v\}\subset S^2$  with points $(a_v,b_v)$ in the corresponding diagrams that continuously change with either $a_v = x\cdot v$ or $b_v=x\cdot v$.  
\end{enumerate}
We will prove these claims later in this proof.

We first define some maps that we will use. Fix a vertex $x$ of $M$,  and a direction $v\in S^2$.
Let $M(v)_{x\cdot v}$ be the sub-level set of $h_v$ from the height $h_v(x)$
and let $M(v)_{x\cdot v}^-$ be the sub-level set of $h_v$ from the the height
$h_v(x)-\delta$ where $\delta>0$ is sufficiently small enough that no
critical values of $h_v$  occur in $(h_v(x)-\delta, h_v(x))$. The
finiteness assumption on the simplicial complex ensures that a suitable $\delta>0$ exists. By the definition of relative homology we have the following exact sequence
\begin{eqnarray*}
&\ldots \to H_{i}(M(v)_{x\cdot v}^-) \xrightarrow{\iota_*} H_{i}(M(v)_{x\cdot v}) \to
H_{i}(M(v)_{x\cdot v}, M(v)_{x\cdot v}^-) \\
&  \xrightarrow{\iota_*} H_{i-1}(M(v)_{x\cdot v}^-) \to H_{i-1}(M(v)_{x\cdot v}) \to \ldots
 \end{eqnarray*}
where $\iota_*$ is the map on homology induced by the inclusion map.
The above implies 
\begin{align}
\begin{split}
 H_i(M(v)_{x\cdot v}, M(v)_{x\cdot v}^-) &=0 \text{ for }i\geq 3\\
 H_{0}(M(v)_{x\cdot v}, M(v)_{x\cdot v}^-) & \simeq \coker \{H_{0}(M(v)_{x\cdot v}^-)\to
 H_0(M(v)_{x\cdot v})\} \\
 H_{2}(M(v)_{x\cdot v}, M(v)_{x\cdot v}^-) \simeq &\coker
 \{H_{2}(M(v)_{x\cdot v}^-)\to H_2(M(v)_{x\cdot v})\}\\
 &\oplus \ker \{H_{1}(M(v)_{x\cdot v}^-)\to H_1(M(v)_{x\cdot v})\}\\
 H_{1}(M(v)_{x\cdot v}, M(v)_{x\cdot v}^-) &  \simeq \coker
 \{H_{1}(M(v)_{x\cdot v}^-)\to 
 H_1(M(v)_{x\cdot v})\}\\
 &\oplus \ker \{H_{0}(M(v)_{x\cdot v}^-)\to H_0(M(v)_{x\cdot v})\}.
\end{split}
\label{eq:ker+cok}
\end{align}
The ranks of the above kernels and cokernels can be read from the appropriate persistence diagrams.
Let $\tilde{\beta}_i(x,v):=\rank(H_{i}(M(v)_{x\cdot v}, M(v)_{x\cdot v}^-)$ denote the relative homology Betti numbers.\footnote{Under nice circumstances (which are always true in this paper's setting, the relative homology groups $H_*(A,B)$ are the same as the reduced homology groups $\tilde{H}_*(A/B)$ where $A/B$ is the set of points in $A$ after we glue all of $B$ together into a single point. Relative homology is almost the same as normal homology except we reduce the dimension by one when looking as the $0$-th dimensional reduced homology. This implies that the $\tilde{\beta}_k(K)=\beta_k(K)$ for $k>0$ and $\tilde{\beta}_0(K)=\beta_0(K)-1$.}  We can compute these relative homology Betti numbers using \eqref{eq:ker+cok}. We have $\tilde{\beta}_i(x,v) =0$ for $i\geq 3$. We have $\tilde{\beta}_2(x,v)$ is the number of classes in $X_2(M,v)$ that are born at height $h_v(x)$ plus the number of classes in $X_1(M,v)$ that die at height $h_v(x)$.
 Similarly $\tilde{\beta}_1(x,v)$ are the number of classes in $X_1(M,v)$ that are born at height $h_v(x)$ plus the number of classes in $X_0(M,v)$ that die at height $h_v(x)$. Finally $\tilde{\beta}_0(x,v)$ is the number of classes in $X_0(M,v)$ that are born at height $h_v(x)$. We will infer the link of $x$ by considering how these ranks vary across $v\in S^2$.

We now prove claim (1); that changes in homology can only happen when a
height function reaches a vertex. This is because if $x$ is not a vertex then for a sufficiently small $\delta>0$ we have $H_k(M(v)_{x\cdot v}, M(v)_{x\cdot v}^-)=0$ for all $k$ and all directions $v$. This lack of homology is reflected in a corresponding lack of points in the persistence diagrams. 

The proof of claim (2) will become apparent later. It is clear that  if $x$ is an isolated vertex then an $H_0$ class is born at height $x_i$ in every direction so we will only need to later prove this claim for vertices that are not isolated. 

%\noindent{\bf Claim (2):} For each vertex $x$, there are only
%a finite number of great circles, each corresponding to another
%vertex, $x'$, for which the height of $x$ and $x'$ are the same. 
%This is a set of measure zero. Outside of this set of great circles \input{cube.tex}

%any changes in the relative homology must be attributable only to the 
%neighborhood of $x$ and hence the ranks of the relative homology Betti 
%numbers correspond to the homology of $(M(v)_{x\cdot v}\cap B(x,r))/
%(M(v)_{x\cdot v}^-\cap B(x,r))$ for suitably small $r$ (we may need to reduce
%$\delta>0$ further). From now on we will ignore sets of measure
%zero. Alternatively we could explicitly ignore all the points on the
%above mentioned great circles. These great circles can be calculated
%from the locations of the vertices. \\

\noindent{\bf Finding vertices:} We now provide a procedure to find the vertices given the above claims.
Both coordinates (when finite) of every point in every persistence diagram must be accounted for. This is how we can guarantee all of the vertices have been found. 
We follow this algorithm repeatedly.
\begin{enumerate}
\item[(1)] Choose a direction $v$, a dimension $k$, and a point $(a_v,b_v)\in X_k(M,v)$ 
\item[(2)] The continuity of $X_k(M,v)$ as $v$ varies ensures that there is a radius $r>0$ such that there is a well defined and continuous set of points $(a_u, b_u)$ for each $u\in B(v,r)$ including the point $(a_v,b_v)$.
\item[(3)] Consider $0<r'<r$. If there exists a point $x\in \R^3$ such that $a_u=x\cdot u$ for all $u\in B(v,r')$ then $x$ is a vertex in $M$. We now have accounted for $a_u$ for $u\in B(v, r')$.
\item[(4)] Consider $0<r'<r$. If there exists a point $x\in \R^3$ such that $b_u=x\cdot u$ for all $u\in B(v,r')$ then $x$ is a vertex in $M$. We now have accounted for $b_u$ for $u\in B(v, r')$.
\end{enumerate} 
This procedure will find all the vertices in the simplicial complex. \\

%\noindent{\bf Finding edges:} We now know the set of vertices. For each pair of vertices $x_i, x_j$ we wish to check whether the edge $[x_i, x_j]$ exists. Consider the great circle of directions for are perpendicular to the vector $x_i-x_j$. We will consider $\tilde{\beta}_k(x_i,v)$ for the paths of $v$ crossing this great circle. Suppose that the edge $[x_i,x_j]$ does not exist. \\

\noindent{\bf Finding links:} Given the set of vertices $\{x_i\}_{i=1}^n$ we need to find the link structure for each vertex to finish the proof. Fix a vertex $x$ .The link of the vertex $x$ is $\partial B(x,r) \cap M$, for a suitable small $r>0$, and then scaled to the unit sphere. Denote the link of $x$ as $L(x)$.

If a vertex $x$ is isolated (i.e. has an empty link) then an $H_0$ class is born at height $x\cdot v$ in every direction $v$ and this point results 
in no other changes in homology. This can be read off the persistence homology transform. From now on suppose that $x$ is not isolated.

We first will wish to find the ``essential'' edges out of $x$. We can consider an edge to be essential if every simplicial representation of $M$ with vertices $\{x_i\}_{i=1}^n$ must contain that edge. For example, the sides of a rectangle are essential but the diagonals are not. For each essential edge out of $x$ we will determine in what directions perpendicular to that edge $M$ exists. From this information, using the piecewise linear structure of $M$, we can piece together the link at $x$.

From now on we will only be considering essential edges and all mention of edges will mean essential edges. We can observe that if $M$ is a simplicial complex whose vertices are in general position then every edge is essential.

%For every non-isolated vertex $x_i$ we will find each of the edges with $x_i$ on their boundary and then find the link of each of these edges. Consider an edge $[x_0,x_1]$. The link of the edge $[x_0,x_1]$ has a vertex $x_j$ for each face  $[x_0,x_1,x_j]$ which has the edge $[x_0,x_1]$ in its boundary. The link of $[x_0, x_1]$ has an edge $[x_j,x_k]$ for each tetrahedron $[x_0,x_1, x_j, x_k]$ which has the edge $[x_0,x_1]$ in its boundary. Note that our assumption that we are working with the simplest representation means that the link of an edge can not be homotopic to all of $S^1$ nor can it be two vertices that are directly opposite.
Recall that $\tilde{\beta}_k(x,v) = \rank(H_{k}(M(v)_{x\cdot v}, M(v)_{x\cdot v}^-))$. Let $$\tilde{\chi}(x,v):=\tilde{\beta}_0(x,v)-\tilde{\beta}_1(x,v) + \tilde{\beta}_2(x,v).$$ 
This is the change in the Euler characteristic from $M(v)_{x\cdot v}^-$ to $M(v)_{x\cdot v}$.
Suppose that $e$ is an edge out of $x$. Without loss of generality we orient $S^2$ (the range of directions in which we consider the corresponding height functions) to have $e$ pointing to the north pole. If $e$ is isolated (i.e. its link is empty) then it contributes a path from $x$ to $M(v)_{x\cdot v}^-$ whenever $v$ is in the southern hemisphere which is not available for any direction in the northern hemisphere. This means that there is a contribution of increasing $\tilde{\beta}_1(x,v) - \tilde{\beta}_0(x,v)$ (and hence decreasing $\tilde{\chi}(x,v)$)  by $1$ as $v$ passes southwards across the equator. This contribution is illustrated in Figure \ref{fig:isolatededge}.

\begin{center}
\begin{figure}[hbt]
\begin{subfigure}{0.5\textwidth}\centering
\begin{tikzpicture}
\draw[fill=lgray, lgray] (-0.25, -1)--(0.5,2)--(0.75,2)--(0, -1);
\draw[thick] (0,0)--(0,2);
\fill  (0,0) node [below left] {$x$} circle (2pt);
%\node at (.2,-.2) {$x$};
\node at (-0.5,1)[above right] {$e$};
\draw (0,0) ellipse (2cm and 0.4cm);
\draw (-0.25, -1)--(0.5,2);
\draw (0, -1)--(0.75,2);
\draw[->] (0,0)--(-2, 0.5) node [above]{$v$};

%\draw (0,0)--(-0.8,1);
\end{tikzpicture}
\end{subfigure}
\begin{subfigure}{0.5\textwidth}\centering
\begin{tikzpicture}
\draw[fill=lgray, lgray] (0.25, -1)--(-0.5,2)--(-0.25,2)--(0.5, -1);
\draw[thick] (0,0)--(0,2);
\fill  (0,0) node [below left] {$x$} circle (2pt);
%\node at (.2,-.2) {$x$};
\node at (0.5,1)[above left] {$e$};
\draw (0,0) ellipse (2cm and 0.4cm);
\draw (0.25, -1)--(-0.5,2);
\draw (0.5, -1)--(-0.25,2);
\draw[->] (0,0)--(-2, -0.5) node [above]{$v$};
\end{tikzpicture}
\end{subfigure}
\caption{As $v$ changes from the northern to the southern hemisphere the edge $e$ is becomes included in $(M(v)_{x\cdot v},M(v)_{x\cdot v}^-)$, which here is indicated by the gray shaded region. The edge $e$ acts as an extra path from $x$ to $M(v)_{x\cdot v}^-$ which increases $\tilde{\beta}_1(x,v) - \tilde{\beta}_0(x,v)$. }\label{fig:isolatededge}
\end{figure}
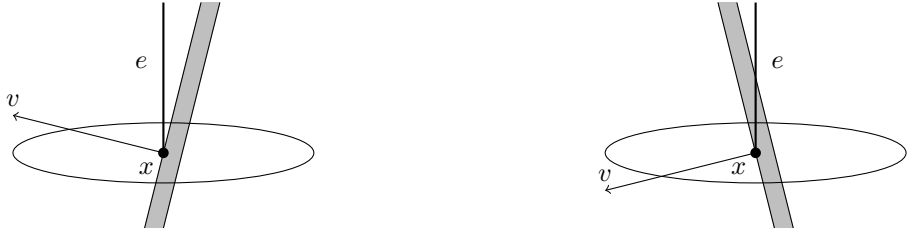
\end{center}
Suppose now that the edge $e$ is not isolated. We need to further consider its link. More precisely we will consider the range of directions perpendicular to $e$ within $M$. Consider the great circle perpendicular to $e$ which is now the equator due to our orientation. Project onto the ball the directions that emanate perpendicularly from $e$. Taking a bird's eye view we can split the equator into regions depending on how many components are in this projection of the link of $e$ intersected with the other half of this equator. This is illustrated in Figure \ref{fig:birdeye}. 
\begin{center}
\begin{figure} [hbt]
\begin{minipage}{0.3\textwidth}\centering
\begin{tikzpicture}
\draw (0,0)--(0,2);
\draw (0,0)--(0.5,1.2)--(0,2);
\node at (.2,-.2) {$x$};
\draw[thick,->] (1,0.5)--(0,0.8);
\draw[fill=gray] (0,0)--(0,2)--(0.5,1.2)--(0,0);
\draw (0,0)--(-0.8,1)--(0,2);
\draw (0,0)-- (-1.5,-1.8)--(0,2);
\draw (-0.8,1)--(-1.5,-1.8);
\draw[fill=gray] (0,0)-- (-1.5,-1.8)--(0,2)--(0,0);
\draw[fill=gray]  (-0.8,1)--(-1.5,-1.8)--(0,2)--(-0.8,1);
\node at (1,0.5)[right] {$e$};
\draw[->] (1,0.5)--(0,0.8);
\draw (0,0) ellipse (2cm and 1cm);
\draw (0,0)--(-0.8,1);
\end{tikzpicture}
\end{minipage}
\begin{minipage}{0.3\textwidth}\centering
\begin{tikzpicture}
\draw (0,0) circle (1.5);
\draw (0,0)-- (285:1.5cm);

%\draw (0,0)-- (250:1.5cm);
%\draw (0,0)-- (130:1.5cm);
\filldraw[fill=gray] (0,0) -- (130:1.5cm) arc (130:250:1.5cm) -- (0,0);
\end{tikzpicture}
\end{minipage}
\begin{minipage}{0.3\textwidth}\centering
\begin{tikzpicture}
\draw (0,0) circle (1.5);
\draw[dashed] (0,0)-- (285:1.5cm);
\draw (0,0)-- (195:2cm);
\draw (0,0)-- (15:2cm);
\draw (0,0)-- (220:2cm);
\draw (0,0)-- (160:2cm);
\draw[dashed] (0,0)-- (250:1.5cm);
\draw[ultra thick] (130:1.5cm) arc (130:250:1.5cm);
\fill (285:1.5cm) circle (2pt);
\draw[dashed] (0,0)-- (130:1.5cm);
\fill[opacity=0.5,fill=gray] (0,0) -- (130:1.5cm) arc (130:250:1.5cm) -- (0,0);
%\draw (160:1.75cm) arc (195:1.75cm);
%(6,0) arc [radius=1, start angle=45, end angle= 120]
\node at (177.5:1.75cm) {$1$};
\node at (207.5:1.75cm) {$0$};
\node at (297.5:1.75cm) {$1$};
\node at (87.5:1.75cm) {$2$};
\end{tikzpicture}
\end{minipage}
\caption{The simplicial complex near $e$, then taking a bird's eye view down $e$, and finally the number of components appearing in the alternate semicircle to that vector over regions of the circle.}\label{fig:birdeye}
\end{figure}
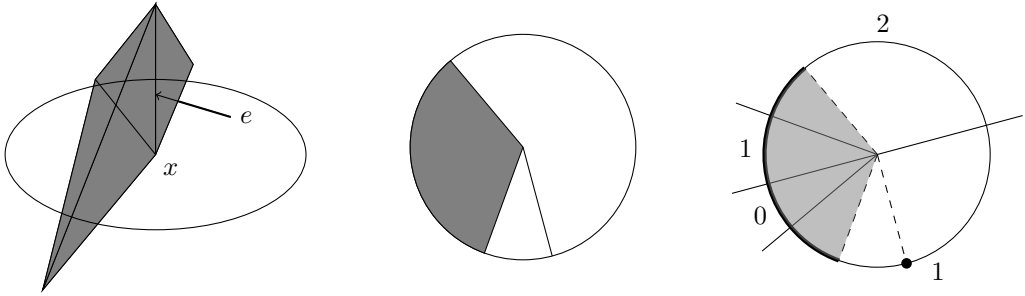
\end{center}
We are interested in how the $\tilde{\chi}(x,v)$ changes as $v$ passes the equator traveling south. There are the following possibilities.
\begin{itemize}
\item[(0)] If there are no components then a new set of paths from $x$ to $M(v)_{x\cdot v}^-$ is born and the $\tilde{\beta}_1(x,v)-\tilde{\beta}_0(x,v)$ is increased by $1$ as $v$ passes southwards. (This is comparable to the isolated edge case as we do not ``see" any part of the link of $e$ and is illustrated in Figure \ref{fig:isolatededge}.)
\item[(1)] If there is one component then there is no change to any of the $\tilde{\beta}_i(x,v)$ as $v$ passes southwards. 
\item[(2)] If there are $2$, then $\tilde{\beta}_2(x,v)-\tilde{\beta}_1(x,v)$ is increased by $1$ as $v$ passes southwards. Either $\tilde{\beta}_2(x,v)$ is increased as two already connected components (connected outside of the link of the edge) join or $\tilde{\beta}_1(x,v)$ is decreased as two formerly unconnected components. Examples of this are illustrated in Figures \ref{fig:2components1} and  \ref{fig:2components2}.
\item[($k$)] If there are $k$ components, $k>1$, then $\tilde{\beta}_2(x,v)-\tilde{\beta}_1(x,v)$ is increased by $k-1$ as $v$ passes southwards. The idea is a generalization of the $2$ component case. We can construct a graph $G$ with one vertex for each of the connected components in Figure \ref{fig:birdeye}. We add edges between these ``connected component" vertices when there is some face in $M$ (with $x$ in its boundary) between them. When $v$ lies below the equator $M(v)_{x\cdot v}/M(v)_{x\cdot v}^-$ is homotopy equivalent to double cone on $G$. For $v$ lying above the equator, we create a different graph $G'$ which is the same as $G$ but we add a vertex representing the edge $e$ and also add $k$ edges one for each of the connected components. These edges go from the vertex representing $e$ to the vertex representing the corresponding connected component. For $v$ lying below the equator, $M(v)_{x\cdot v}/M(v)_{x\cdot v}^-$ is homotopy equivalent to the double cone on $G'$. As $v$ passes southwards we effectively glue one edge and $k$ discs to get from the double cone on $G$ to the double cone on $G'$. This increases $\tilde{\beta}_2(x,v)-\tilde{\beta}_1(x,v)$ is by $k-1$. For example, in the case of $k=2$, we have to possible graphs for $G$; either two disconnected vertices (as shown in Figure \ref{fig:2components1}) or two vertices connected by one edge (as shown in Figure \ref{fig:2components2}). 
\end{itemize}
\begin{center}
\begin{figure}[hbt]
\begin{subfigure}{0.5\textwidth}\centering
\begin{tikzpicture}
\draw[pattern=crosshatch] (-0.25, -1)--(0.5,2)--(0.75,2)--(0, -1);
\draw[fill=lgray] (0,0)--(0,2)--(1,1.5)--(0,0);
\draw[fill=lgray] (0,0)--(0,2)--(0.5,0.5)--(0,0);
\draw[pattern=crosshatch] (-0.25, -1)--(0.5,2)--(0.75,2)--(0, -1);
\draw[thick] (0,0)--(0,2);
\fill  (0,0) node [below left] {$x$} circle (2pt);
%\node at (.2,-.2) {$x$};
\draw (0,0) ellipse (2cm and 0.4cm);
\draw (-0.25, -1)--(0.5,2);
\draw (0, -1)--(0.75,2);
\draw[->] (0,0)--(-2, 0.5) node [above]{$v$};

%\draw (0,0)--(-0.8,1);
\end{tikzpicture}\caption{ }%$M(v)_{x\cdot v}/M(v)_{x\cdot v}^-$ is homotopy equivalent to circle.}
\end{subfigure}
\begin{subfigure}{0.5\textwidth}\centering
\begin{tikzpicture}
\draw[pattern=crosshatch] (0.25, -1)--(-0.5,2)--(-0.25,2)--(0.5, -1);
\draw[fill=lgray] (0,0)--(0,2)--(1,1.5)--(0,0);
\draw[fill=lgray] (0,0)--(0,2)--(0.5,0.5)--(0,0);
\draw[pattern=crosshatch] (0.25, -1)--(-0.5,2)--(-0.25,2)--(0.5, -1);
\draw[thick] (0,0)--(0,2);
\fill  (0,0) node [below left] {$x$} circle (2pt);
%\node at (.2,-.2) {$x$};
\draw (0,0) ellipse (2cm and 0.4cm);

\draw[->] (0,0)--(-2, -0.5) node [above]{$v$};

%\draw (0,0)--(-0.8,1);
\end{tikzpicture}\caption{ }% $M(v)_{x\cdot v}/M(v)_{x\cdot v}^-$  is homotopy equivalent to a point.}
\end{subfigure}\caption{In (a), when $v$ is above the equator, $M(v)_{x\cdot v}/M(v)_{x\cdot v}^-$ is homotopy equivalent to circle. The addition of $x$ into the sublevel sets  creates a loop. In (b), when $v$ is below the equator, $M(v)_{x\cdot v}/M(v)_{x\cdot v}^-$  is homotopy equivalent to a point. The inclusion of $x$ is a continuous variation and hence does not change the homology. As $v$ passes the equator southwards we switch between these cases and $\tilde{\beta}_1(x,v)$ decreases by one.}\label{fig:2components1}
\end{figure}
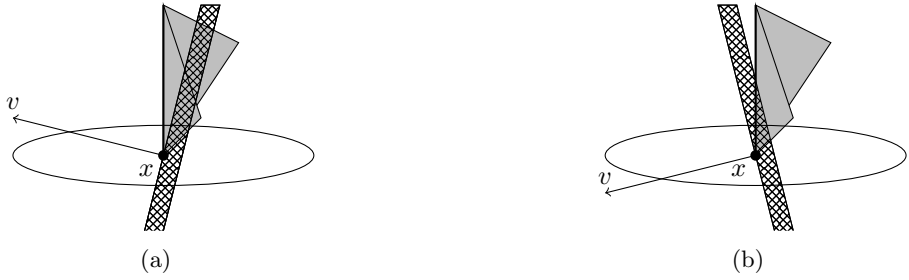
\end{center}
\begin{center}
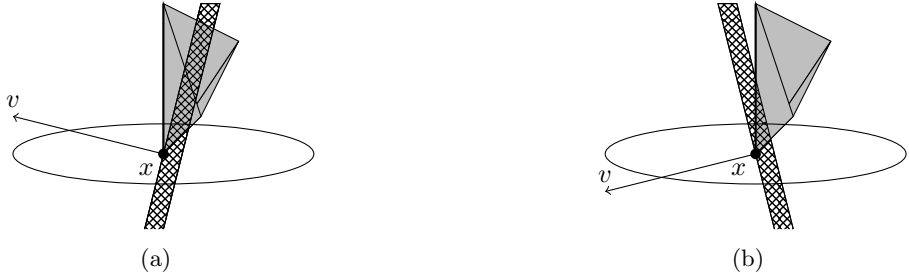
\begin{figure}[hbt]
\begin{subfigure}{0.5\textwidth}\centering
\begin{tikzpicture}
\draw[pattern=crosshatch] (-0.25, -1)--(0.5,2)--(0.75,2)--(0, -1);
\draw[fill=lgray] (0,0)--(0.5,.5)--(1,1.5)--(0,0);
\draw[fill=lgray] (0,0)--(0,2)--(1,1.5)--(0,0);
\draw[fill=lgray] (0,0)--(0,2)--(0.5,0.5)--(0,0);
\draw[pattern=crosshatch] (-0.25, -1)--(0.5,2)--(0.75,2)--(0, -1);
\draw[thick] (0,0)--(0,2);
\fill  (0,0) node [below left] {$x$} circle (2pt);
%\node at (.2,-.2) {$x$};
\draw (0,0) ellipse (2cm and 0.4cm);
\draw (-0.25, -1)--(0.5,2);
\draw (0, -1)--(0.75,2);
\draw[->] (0,0)--(-2, 0.5) node [above]{$v$};

%\draw (0,0)--(-0.8,1);
\end{tikzpicture}\caption{}% $M(v)_{x\cdot v}/M(v)_{x\cdot v}^-$ is homotopy equivalent to point.}
\end{subfigure}
\begin{subfigure}{0.5\textwidth}\centering
\begin{tikzpicture}
\draw[pattern=crosshatch] (0.25, -1)--(-0.5,2)--(-0.25,2)--(0.5, -1);
\draw[fill=lgray] (0,0)--(0,2)--(1,1.5)--(0,0);
\draw[fill=lgray] (0,0)--(0.5,.5)--(1,1.5)--(0,0);
\draw[fill=lgray] (0,0)--(0,2)--(0.5,0.5)--(0,0);
\draw[pattern=crosshatch] (0.25, -1)--(-0.5,2)--(-0.25,2)--(0.5, -1);
\draw[thick] (0,0)--(0,2);
\fill  (0,0) node [below left] {$x$} circle (2pt);
%\node at (.2,-.2) {$x$};
\draw (0,0) ellipse (2cm and 0.4cm);

\draw[->] (0,0)--(-2, -0.5) node [above]{$v$};

%\draw (0,0)--(-0.8,1);
\end{tikzpicture}\caption{ }%$M(v)_{x\cdot v}/M(v)_{x\cdot v}^-$  is homotopy equivalent to a 2-sphere.}
\end{subfigure}\caption{In (a), when $v$ is above the equator, $M(v)_{x\cdot v}/M(v)_{x\cdot v}^-$ is homotopy equivalent to point. The addition of $x$ into the sublevel sets is a continuous variation and hence does not change the homology. In (b), when $v$ is below the equator, $M(v)_{x\cdot v}/M(v)_{x\cdot v}^-$ is homotopy equivalent to a 2-sphere. The inclusion of $x$ into the sublevel sets completes the outside of a void. As $v$ passes the equator southwards we switch between these cases and  $\tilde{\beta}_2(x,v)$ increases by one.}\label{fig:2components2}
\end{figure}
\end{center}
Together we see that the link at $e$ causes $\tilde{\chi}(x,v)$ to increase by $k-1$ if the link of $e$ intersected with the semicircle furtherest from $v$ contains $k$ components. For each edge $e$, let $f_e$ denote the function on the great circle perpendicular to $e$ of the changes in $\tilde{\chi}(v,x)$ as we pass ``southwards'' over the equator away from $e$.  Knowing the function $f_e$ is equivalent to knowing, at each location, how many components lie in the alternate semicircle. This in turn is equivalent to knowing the birds eye picture as illustrated in Figure \ref{fig:birdeye} which is in turn equivalent to knowing the link of $e$.

We cannot make any comments about what happens on the equator itself. However there are only finitely many vertices and so there are only finitely many edges. In turn this implies that the set of directions perpendicular to an edge at $x$ is of measure zero. From now on we only consider functions up to sets of measure zero.

The sphere of directions can be partitioned into regions bounded by finitely many great circles perpendicular to edges emanating from $x$. Within the same region the $\tilde{\chi}$ remains constant. The above process shows how $\tilde{\chi}$ varies as it passes between regions.  Consider an edge $e$. From our assumption that we are considering an essential edge we know that the number of components is not $1$ for some open interval along the great circle perpendicular to $e$. This implies that at least one region bounded by great circles has non-zero $\tilde{\chi}$ and hence $x$ determines a critical point for directions in that region. This proves claim (2).

We now describe how we scan all the vertices and find their links. \\
\noindent{(1)} Select a direction $v_0$ for which no vertices have the same height in that direction. We will iterate the following procedure through all the vertices 
in order of their height in the direction of $v_0$. This is possible as we are only considering simplicial complexes with finitely many vertices. We will scan through  the vertices in this direction, building up the complex as new vertices are included. It is useful that at each stage, when we want to find the link of a new vertex, that we already know its neighborhood intersected with the half plane in the direction $-v_0$.  For the base case, we know that we first hit the simplicial complex at some finite time, due to the finiteness condition. The neighborhood of this first vertex $x$ intersected with $M(x,v_0)$ is only the point $x$ itself. \\
\noindent{(2)} We now investigate vertex $x$. We know the sublevel set $M(x,v_0)$.  Consider the partition of the sphere around $x$ into regions with the same relative homology. There is a possibility that there are edges  $e$ and $e'$ out of $x$ directly opposite to each other with links such that the effects of the these links passing over the equator cancel. We can remedy this situation by  including in our list of great circles those perpendicular to edges next to $x$ lying in $M(x,v_0)$. This partition tells us which great circles  corresponding to edges exist. \\
\noindent{(3)}  Consider a great circle $C$ found in (2). It has perpendicular normals $\eta$ and $-\eta$ with $v_0 \cdot \eta <0$. From $v_0\cdot \eta <0$ we know that if there was an edge in the direction of $\eta$ then it would be in $M(x,v_0)$. Furthermore we would know its link. \\
\indent{(3a)} If there is no edge in the direction of $\eta$ then there must be an edge $e$ in the direction of $-\eta$. We also know that $f_e=-f$. The minus sign is because of switching the orientation so that $e$ is pointing the to north pole instead of the south. Sine we know $f_e$ we can determine the link of $e$.\\
\indent{(3b)} If there is an edge $e'$ in the direction of $\eta$ consider the new function $g=f-f_{e'}$. If $g$ is the zero function (recall everything is up to sets of measure zero) then every change of $\tilde{\chi}$ as vector pass the great circle $C$ can be attributable to $e'$ and hence there is no edge in the direction $-\eta$. If $g$ is not the zero function then there is an edge $e$ in the direction of $-\eta$. As in the case of (3a) we know that $f_e=-g$ and we can thus determine the link of $e$.

\noindent{(4)} For each vertex, in the order outline in (1), we first find the appropriate great circles by step (2). We then iterate step (3) through all great circles.  Remember each iteration will assign changes to vertices and/or links and at the next iteration we ignore previously labeled vertices and links. When we have assigned all the changes we will have revealed 
the entire simplicial complex, all the links and vertices.

\end{proof}

Although it is possible to write a direct proof for the injectivity of the persistent homology transform for simplicial complexes in the the plane it is easier and faster to consider it a corollary of the three dimensional case.

\begin{corollary}
The persistent homology transform is injective when the domain is $\M_2$.
\end{corollary}
\begin{proof}
Let us consider $\R^2$ as being inside $\R^3$ with the third coordinate set to zero. This means that we can think of $\M_2$ as lying inside $\M_3$. 
Consider $M \in \M_2$ as a simplicial complexes in $\R^3$. 
Let $v=(v_1, v_2, v_3)$ be a unit vector in $\R^3$ with $v_3 \neq \pm 1$. Let $(\tilde{v}_1, \tilde{v}_2)$ be the unit vector in the direction of $(v_1, v_2)$. We have $\tilde{v}_i \| (v_1, v_2) \|=v_i$.
Now $M((v_1, v_2, v_3))_r$ is the set of all points $(x,y,z)$ in $M$ (viewed as a subset of $\R^3$) such that $(x,y,z)\cdot(v_1, v_2,v_3)\leq r$. Now   $(x,y,z)\in M$ implies that $z=0$ and hence  $M((v_1, v_2, v_3))_r$ is the set of $(x,y)$ in $M$ (viewed as a subset of $\R^2$) such that $(x,y) \cdot (v_1, v_2) \leq r$. 
Now $(x,y) \cdot (v_1, v_2) \leq r$ is equivalent to $(x,y) \cdot (\tilde{v}_1, \tilde{v}_2) \leq r/\|(v_1, v_2)\|$ so we can see that $M((v_1,v_2, v_3))_r$ is in fact $M((\tilde{v}_1, \tilde{v}_2))_{ r/\|(v_1, v_2)\|}$.
This implies that we can construct $X_p(M, (v_1, v_2, v_3))$ from $X_p(M, (\tilde{v_1}, \tilde{v_2}))$ by appropriately scaling the points in the persistence diagram.

If $v=(0,0,\pm1)$ then $M(v)_r$ is empty for $r<0$ and is $M$ for $r\geq 0$. This tells us that the persistence diagrams simply contain a set of points at $(0, \infty)$ to represent the homology of $M$.

Finally note that $X_3(M,v)$ is the empty diagram for all directions $v \in S^2$.

Let $M_1, M_2 \in \M_2$. If $\PHT(M_1) : S^1 \to \D^2$ is the same as $\PHT(M_2): S^1 \to \D^2$ then our above construction process shows that $\PHT(M_1): S^2 \to \D^3$ is the same as $\PHT(M_2):S^2 \to \D^3$ when both $M_1$ and $M_2$ are embedded in $\R^3$ by setting the third coordinate to be zero. Now the persistent homology transform is injective on $\M_3$ by Theorem \ref{inject}. This implies that $M_1=M_2$.
\end{proof}

A result of the above theorem and corollary is that we can model the space of piece-wise linear simplicial complexes in $\R^3$ (or $\R^2$) by modeling the images of their persistent homology transforms which lie inside $C(S^2,\D^3)$ (or $C(S^1, \D^2)$ respectively).
We can define distances between two shapes $M$ and $M'$ by the distance between $\PHT(M)$ and $\PHT(M')$ -- we can pull back the metric on the space 
of diagrams to a metric on the space of piece-wise linear objects in $\R^3$ (or $\R^2$). We can also specify a likelihood over shapes, which is difficult, by  defining
a likelihood of a collection of points. We can use point processes for a likelihood model over PHT space.

%%%UP TO HERE

We now use the above result to prove sufficiency of the PHT. 
\begin{cor}\label{sufdist}
Consider the subspace of shapes $\mathcal{M}_k^N$ (for  $k=2$ or $3$), piecewise linear simplicial complexes with at most $N$ vertices. Let $f(x;\theta)$ be a density function over $\M_k$ with parameters $\theta \in \Theta$ and $x \in \mathcal{M}_k$ whose support is contained in some $\M_k^N$. Then the persistence homology transform $t(X) \in  C(S^{k-1},\D^k)$ is a sufficient statistic.
\end{cor}

\begin{proof}
Denote the subset $C(S^{k-1}, \D^k)$ that is realizable by the PHT applied to $\mathcal{M}_k^N$ as $t \in {\cal T}$. \\

We first state the Fisher-Neyman factorization theorem \cite{Neyman35}. Given a joint density function $f(x;\theta), \, \theta \in \Theta$ then a statistic $T = T(X)$ is sufficient for 
$\theta$ if and only if 
$$f(x;\theta) = g(t(x), \theta) \, h(x),$$
where $g(\cdot)$ and $h(\cdot)$ are functions. A more rigorous version
of the above result with respect to measure theory was given by Halmos
and Savage \cite{HalmosSavage49}. This version of the theorem follows:
A necessary and sufficient condition that a statistic $T$ be
sufficient for a dominated set of measures $\mathscr{K}$($ \ll \lambda$) on $\mathcal{M}_k$  is that for every $\mu \in \mathscr{K}$ 
the density $f_\mu = \frac{d\mu}{d\lambda_0}$ can be factorized as
$$f_\mu(x) = g_\mu(T(x)) \cdot h(x)$$
and that $g_\mu$ is a $\cal T$ measurable function and $h(x) \neq 0$ is a $\mathcal{M}_k$ measurable function. We include this version of stating sufficiency via factorization to allay measure theoretic concerns about the Fisher-Neyman version.

From the injectivity statement in Theorem \ref{inject} we know there exist functions 
$$h:  \mathcal{M}_k \to {\cal T}, \mbox{ and  } \ell : {\cal T} \to \mathcal{M}_k.$$
We use the notation
$$g(t(x),\theta) \equiv g_\theta(t(x)), \mbox{ and  } f(x;\theta) \equiv f_\theta(x).$$
The following relations show that the condition of the Fisher-Neyman factorization theorem holds for the PHT
\begin{eqnarray*}
f_\theta(x) & = & f_\theta(\ell(t(x))), \\
                    & = &f_\theta \circ \ell(t(x)),\\
                    & = & g_\theta(t(x)),
\end{eqnarray*}
where $g_\theta =  f_\theta \circ \ell.$

We now want to verify that all the relevant functions are measurable. In our case the function $h(x)$ will be constant and thus it is automatically measurable. In order to show that $g_\theta$  is ${\cal T}$ measurable we observe that $g_\theta=f \circ \ell$. Since by assumption $f$ must be measurable it will be sufficient to show $\ell:\mathcal{T} \to \M_k^N$ is measurable. We use the Borel sigma algebras associated with certain distance functions on $\M_k^N$ and $\cal T$. On $\mathcal{T}$, consider the $L_1$ function distance in $C(S^{d-1}, \D^d)$, using the $L_1$ distance in $\D$.

There are only finitely many different possible simplicial complexes on $N$ labelled vertices. Given a simplicial complex $K$ on $N$ labelled vertices, the possible maps $f:K \to \R^d$ such  that the restriction of $f$ to each simplex is linear is determined by the locations of the $N$ vertices in $\R^d$. This implies that the space of possible maps lives in $\R^{d\times N}$. The subset of maps where the preimage under $f$  of every point in $f(K)$ is starlike is an open subset. There is a natural distance for this subset in  $\R^{d\times N}$ inherited from Eucildean distance - denote this distance by $d_K(f_1, f_2)$. Define the distance function $d'$ over $\M_d$ as follows. Let $M_1, M_2 \in \M_d$ and with slight abuse of notation let $M_i$ also denote the image in $\R^d$ which determines the equivalence class of $M_i$. Each $M_i$ is represented by many different pairs $(K,f)$.  
Set
$$d'(M_1, M_2):=\inf\{ d_K( f_1, f_2): f_1(K)=M_1 \text{ and } f_2(K)=M_2\}.$$
Observe that if $M_1$ and $M_2$ are not homotopy equivalent then the distance between them is infinite as no $K$ exists such that $f_1(K)=M_1$ and $f_2(K)=M_2$.

The PHT is Borel continuous with respect to the these distance functions. Since the inverse of an injective Borel continuous map is Borel continuous we can conclude that $\ell$ is Borel measurable. 
\end{proof}

\subsection{Exponential family models and Euler characteristics}
\label{expfamily}

In statistical modeling the relevance of a sufficient statistic is often through the existence of an exponential family model. An exponential family can be defined as a collection of probability densities with a $d$ dimensional sufficient statistic $T(z) = (T_1(x),...,T_d(x))^T$
such that
$$p_\theta(x) = a(\theta) \, h(x) \, \exp(- \langle \theta, T(x) \rangle),$$
with $\langle \cdot, \cdot \rangle$ as the standard $d$-dimensional inner product. This allows for a likelihood function for observations
of surfaces with the likelihood of the data, $\mbox{Data} \equiv (X_1,...,X_n) \stackrel{iid}{\sim} p_\theta$, stated as
$$\mbox{Lik}(\mbox{Data} \mid \theta ) = \prod_{i=1}^n \Big[  a(\theta) \, h(x_i) \, \exp(- \langle \theta, T(x_i) \rangle ) \Big],$$
where the parameters are associated with the sufficient statistics. Formulating an exponential family model using sufficient statistics that are collections of persistence diagrams is problematic. The complex geometry of the space of persistence diagrams \cite{Turneretal13} is not conducive to a Euclidean inner product structure.

There is a variation of the PHT that is an injective map and has a simple inner product structure. Given the previous height function $$M(v)_r = \{ \Delta\in M: x\cdot v \leq r \text{ for all } x\in \Delta\}$$ 
the Euler characteristic curve $\chi(M,v)$ is the following function of the Euler characteristic
for the subcomplex at values $r$, $ \chi(M,v)(r) = \chi(M(v)_r)$. The Euler characteristic of a subcomplex which in our case is a
a surface of a polyhedra has a simple form 
$$ \chi(M(v)_r) = V - E + F,$$
where $V,E,F$ are the number of vertices, edges, and faces respectively of the subcomplex $M(v)_r$. Based on the Euler characteristic and height functions we can define the Euler characteristic transform (ECT) for 
shapes and surfaces
\begin{align*}
 \ECT(M): S^{d-1} &\to \Z^\R\\
 v&\mapsto (\chi(M,v)).
 \end{align*}
A direct consequence of the proof of Theorem \ref{inject} is that the Euler characteristic transform (ECT)  is also injective. We thus can show by an analogous proof to that in Corollary \ref{sufdist} that the ECT is a sufficient statistic for shapes and surfaces. 

The sufficient statistic is now a collection of curves which can have
a simple inner product structure. We can rescale the domain of the 
Euler characteristic curves to be in the interval $[-1,1]$. Assume for
purposes of computation we use a $K$ vectors sampled from $S^{d-1}$, we now have $K$ curves on the unit interval which is much easier to work with than a persistence diagram. Denote the Euler characteristic curve for a given direction  $f=\chi(M,v)$ over the 
interval $[-1,1]$ we smooth the Euler characteristic curve by constructing the following cumulative curve
$F(x) = \int_{0}^x f(y) dy,$ see Figure 6. The resulting transform is a collection of $K$ smooth curves $\{F_1,...,F_K\}$.

\begin{figure}[ht]
\begin{center}
\begin{tabular}{ccc}
\begin{tabular}{c}
\includegraphics[width=1.3in]{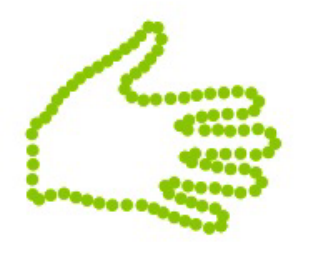} 
\end{tabular}
&  
\begin{tabular}{c}
\includegraphics[width=1.3in]{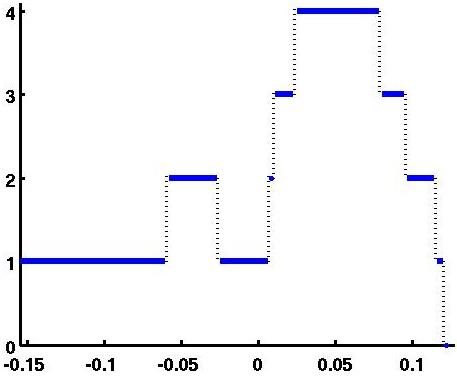} 
\end{tabular}
 &  
\begin{tabular}{c}
\includegraphics[width=1.3in]{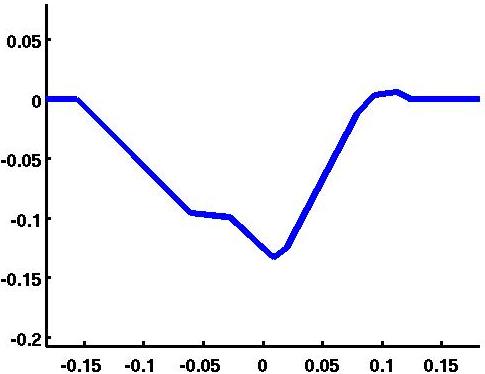} 
\end{tabular}
\\
 \qquad (a)  &  \qquad (b) &  (c) 
\end{tabular}
\caption{(a) A 2D contour of a hand (b) EC curve of the 2D contour of a hand and (c) the associated (centered)
cumulative EC-curve.}
\label{F:eccurves}
\end{center}
\end{figure}

Given the $K$ smooth curves we can define an exponential family model of the form
$$p_\theta(x) = a(\theta) \, h(x) \, \exp\left(- \sum_{k=1}^K \langle \theta, F_k(x) \rangle\right).$$ 
For computational reasons we  sample the curves $\{F_k\}_{k=1}^K$ at $T$ points in the interval $[-1,1]$.
We can now think of the transform of a shape as matrix $\mathbf F$ with ${\mathbf F}_{ij}$ the function value of the
$j$-th point in the $i$-th Euler characteristic curve. A common density function for matrices is the matrix variate normal \cite{Dawid81}
which is a generalization of a multivariate normal  which for the $K\times T$ matrix $\mathbf F$ is
$$p({\mathbf F} \mid {\mathbf A}, {\mathbf U}, {\mathbf V}) = \frac{\exp\left(-\frac{1}{2}  \mbox{tr}[  {\mathbf V}^{-1}({\mathbf F} - {\mathbf A})^T {\mathbf U}^{-1} ({\mathbf F} - {\mathbf A})]\right)}{ (2\pi)^{KT/2} |{\mathbf V}|^{L/2} |{\mathbf U}|^{K/2}},$$
where the parameter ${\mathbf A}$ is the mean matrix, the parameter ${\mathbf U}$ is a $K \times K$ covariance matrix modeling the covariance between curves , and the parameter ${\mathbf V}$ is a $T \times T$ covariance matrix modeling the covariance between
points in the Euler characteristic curve. 

Using the ECT and the matrix variate model, given $n$  meshes $(M_1,...,M_n)$ we can define a likelihood model
\begin{equation}
\label{likmdl}
\mbox{Lik}(M_1,...,M_n \mid {\mathbf A},  {\mathbf U}, {\mathbf V} ) = \prod_{i=1}^n p({\mathbf F}(M_i) \mid  {\mathbf A},  {\mathbf U}, {\mathbf V} ),
\end{equation}
with parameters $\theta = \{   {\mathbf A},  {\mathbf U}, {\mathbf V}  \}$ and ${\mathbf F}(M_i)$ is the matrix constructed from the ECT of 
a mesh $M_i$. This likelihood model can serve as an alternative to landmark based statistical models.

\subsection{Surfaces homeomorphic to spheres}\label{homeomorphic}

We often have further structure for the set of simplicial complexes of interest, such as they are homeomorphic to a sphere. A common example is the surface of a
solid contractible object -- e.g. the boundaries of many physical objects. In Section \ref{results} we examine the calcaneus  or heel bone of
various species. The boundaries of these bones are homeomorphic to $S^2$.

\begin{cor}\label{cor:inject}
Let $\mathcal{S}(d,k)$ be the space of piecewise linear surfaces in $\R^d$ that are homeomorphic to $S^k$. The $0$-th dimensional persistent homology transform is injective when the domain is either $\mathcal{S}(3,2)$ or $\mathcal{S}(3,1)$ or $\mathcal{S}(2,1)$.
\end{cor}

\begin{proof}
Let us first consider the cases where the domain is either $\mathcal{S}(3,1)$ or $\mathcal{S}(2,1)$. It is sufficient to show that from $\PHzeroT(M)$ we can deduce the persistence diagrams of dimensions $1$ as all higher dimensional homology classes are always zero. Pick a direction $v$. Since $M$ is homeomorphic to a sphere the only $H_1$ class is born exactly when the entire loop is revealed. This is the time that the loop is first hit from the direction $-v$. Since we know the $0$-th dimensional persistence classes for direction $v$ we know at what height this is.

Let us now consider the case where the domain is $\mathcal{S}(3,2)$. It is sufficient to show that from $\PHzeroT(M)$ we can deduce the persistence diagrams of dimensions $1$ and $2$. 
Pick a direction $v$.
We know the $0$ dimensional persistence classes by assumption. 
Since $M$ is homeomorphic to a sphere the only $H_2$ class is born exactly when the entire surface is revealed. This is the time that the surface is first hit from the direction $-v$. Since we know the $0$-th dimensional persistence classes for direction $v$ we know at what height this is.

To find the $H_1$ persistent homology classes we will use Alexander duality. We will need to use persistent cohomology which is very similar to persistent homology but the induced maps on cohomology go in the opposite direction. The important fact we will use is that persistent homology and persistent cohomology of the same filtration result in the same persistence diagram \cite{de2011dualities}. 

Now $M$ by assumption is (homeomorphic to) a sphere and $M(v)_h$ is a compact, locally contractible subset of the $M$. By Alexander duality $H_1(M(v)_h)$ is isomorphic to $\tilde{H}^{0}(M(v)_{h}^c)$ where $\tilde{H}^*$ is reduced cohomology and $A^c$ denotes the complement of $A$ in $M$. The reduced homology means that we ignore the essential $H^0$ persistence class. Now $M(v)_{h}^c$ is homotopy equivalent to $ M(-v)_{-h-\epsilon}$ for sufficiently small $\epsilon$. These isomorphisms are compatible with the induced maps from inclusions. For $h_1\leq h_2$ we have the following commutative diagram where the vertical homomorphism are the homomorphisms induced by inclusion:

$$
\xymatrix{
H_1(M(v)_{h_1}) \ar[d]  \ar[r]^\sim 	&\tilde{H}^0(M(v)_{h_1}^c)=\ar[d]\tilde{H}^0(M(-v)_{-h_1}^-)\\
H_1(M(v)_{h_2} )\ar[r]^\sim   		&\tilde{H}^0(M(v)_{h_2}^c)=\tilde{H}^0(M(-v)_{-h_2}^-)}
$$	

This implies that every $H_1$ persistent homology class for direction $v$ that is born at $a$ and dies at $b$ gets sent under these isomorphisms to an $H^0$ persistent homology class for direction $-v$ that is born at $-b$ and dies at $-a$. Since the persistence diagrams computed by cohomology and homology are the same, we can compute the $H_1$ persistence diagrams by flipping the persistence diagrams for $H_0$ without the essential class.
\end{proof}

\subsection{Unaligned objects and shape statistics }\label{align}

A classic framework for modeling shapes and surfaces is that of shape statistics \cite{Bookstein97,Kendall77,Kendall84} where a set of $k$ locations or landmarks
on a $d$-dimensional object (typically considered a manifold) are fixed with $k > d$ and the data consist of the points at these $k$ landmarks, a 
$k$-ad \cite{Bhattacharya2008}. A central idea in shape statistics \cite{Bookstein97,Kendall77,Kendall84} is that $k$-ads should be compared modulo a group of transformations
given by how the data are generated. Typically, these transformations are size or scaling, rotation, and translation.

We now describe how we can adapt our methodology to account for invariance with respect to scaling, translation, and rotation. We are given $n$ objects
$\{M_1,...,M_n\}$, either all in $\R^2$ or all in $\R^3$, and we proceed in three steps: (1) we center the objects, (2) we scale the objects, (3) for each pair of objects we consider all the distances under different rotations and take the smallest of these distances. \\

\noindent {\bf Centering}: Fix a set of equally spaced directions $\{v_i\}$ (or approximately evenly spaced in for objects in $\R^3$) . We will first give a procedure to center an object at the origin with respect to these directions. We are effectively centering the convex hull of the object. For each direction set $\lambda_i$ as the time the first component on the shape is seen
in direction $v_i$ -- that is $\lambda_i$ is the smallest $x$ such that $(x,\infty) \in X_0(M,v_i)$. 
Let $\lambda^u_i$ denote the scalar when the origin is at $u$. (This is the same as the value obtained by taking the vectors at the normal origin and 
shifting $M$ by $u$ giving $M-u = \{x-u:x \in M\}.$) The $v_i$ are unit vectors and $\lambda_i^u$ are signed perpendicular distances to the same hyperplane which has normal $v_i$. The sign is negative if $u$ is on the $v_i$ side of the hyperplane and are positive if $u$ is on the $-v_i$ side of the hyperplane.

Consider two different potential centers $u$ and $w$.  Considering the $\lambda_i$ as signed distances implies that $\lambda^u_i - \lambda^w_i = (w-u)\cdot v_i$ and hence
\begin{align}\label{eq:centers}
\sum_i \lambda^u_i v_i- \sum_i \lambda^w_i v_i= \sum_i(\lambda^u_i - \lambda^w_i )v_i= \sum_i [(w-u)\cdot v_i] v_i=K(w-u).
\end{align}
 where $K$ is a constant independent of $w$ and $u$ so long as the $v_i$ are symmetric with respect to some basis set of vectors.  This $K$ can be easily computed given the specific set of directions of $v$.

We will define the center to be the point $u$ such that $\sum_i \lambda^u_i v_i=0$. The equation \eqref{eq:centers} shows that this is unique and that this center is computed (starting with $w$ as the origin) by
$u= \frac{1}{K}\sum_i \lambda_i v_i$

We center $M$ by
shifting it by $u$ 
$${\tilde M} = M-u = \{x-u:x \in M\}.$$ 
The PHT of $\tilde M$ is related to that of $M$ in that for each persistent homology class at $(a,b) \in X_p(M,v)$
we have the same persistent homology class at $(a-u \cdot v,b-u \cdot v) \in X_p({\tilde M},v)$. For simplicity of notation we rename the centered
object $M$, $M \leftarrow {\tilde M}$. We apply this procedure to each object. \\

\noindent {\bf Scaling}: Set an arbitrary scale parameter $C$, e.g. $C=1$.  Compute $L= - \sum_i \lambda_i>0$ -- these are the same $\{\lambda_i\}$ used in the centering procedure. We now rescale $M \leftarrow \frac{C}{L} M = \{\frac{C}{L} x : x \in M\}$, for the scaled object $\sum_i \lambda_i  =C.$ This is done for each object. \\

\noindent {\bf Rotating}: For each pair $M_j, M_k$ of centered and scaled objects we now consider the different distances under different rotations. We set a subgroup in the group of rotations $\{R_i\}_{i=1}^G$.
Set the unaligned distance between $M_j$ and $M_k$ to be $$\inf_{i=1,...,G} \d_{\M_d}(M_j, R_i(M_k)).$$
This unaligned distance gives a metric on unaligned objects in $\R^d$

In the case of objects in $\R^2$ we can take the $\{v_i\}$ to be evenly spaced unit vectors. The set of rotations can be $\{R_k: R_k v_i =v_{k+i}\}$. In this case we do not need to compute any more persistence diagrams - we just relabel the ones that have already been computed.

In Section \ref{sil} we will applying the above procedure to the bounding circles of a set of silhouette images in the plane. 
 
\section{Results real and simulated data}\label{results}

To illustrate the utility of the PHT we look at two related problems, computing the pairwise distance between a set of aligned objects, and comparing unaligned objects. Before stating results on a data set of shapes and real data consisting of bones we first state the algorithm used to compute distances between objects.

\subsection{Distance algorithm}\label{subsec:alg}

To compute the PHT of an object we need to compute the persistence diagrams of the height function from various directions. For purposes of computation we will use a finite set of vectors sampled from $S^{d-1}$ and average the distance between diagrams of two objects, this serves a a numerical approximation of the distance defined in \eqref{distance}. 

Let $O_1, \ldots O_N$ be the objects we wish to compare. Let $v_1, v_2, \ldots v_K$ be the normal vectors we use. Let $f_{ik}$ be the height function on $O_{i}$ in the direction $v_k$. Let $X(f_{ik})$ be the persistence diagram constructed using sublevel sets of $f_{ik}$. The following pseudocode states the algorithm that computes distances: \\  

\noindent Distance computation algorithm \\
\noindent  {\bf Data}: Objects $O_1, \ldots O_N$,  $K$ directions $v_1,..,v_K$\\
\noindent {\bf Results}: Pairwise distances \\
\noindent  initialization - all pairwise distances $d_{ij}=d(O_i,O_j)$ set to $0$;\\
\noindent For $i=1$ to $N$\\
\indent For $k=1$ to $K$\\
\indent \indent Compute $X(f_{ik})$;\\ \\
\noindent For $i=1$ to $N$; \\
\indent For $j\leq i$ \\
 \indent \indent For $k=1$ to $K$ \\
\indent \indent \indent $d_{ij} \stackrel{+}{=} d(X(f_{ik}),X(f_{jk}))$;\\
\indent \indent $d_{ij}=\frac{d_{ij}}{K}$; \\ \\

% \KwData{$N$ Objects $O_1, \ldots O_N$,  $K$ directions $v_1,..,v_k$}
% \KwResult{Pairwise distances}
% initialization - all pairwise distances $d_{ij}=d(O_i,O_j)$ set to $0$\;
% \For{$i=1$ to $N$}
%	{\For{$k=1$ to $K$}
%		{Compute $X(f_{ik})$}
% 	}
%	
%  \For{$i=1$ to $N$}
%	{\For{$j\leq i$}
%		{\For{$k=1$ to $K$}
%		{$d_{ij} \mathrel{+}=d(X(f_{ik}),X(f_{jk})$)}
%		$d_{ij}=\frac{d_{ij}}{K}$}
% 	}
%% \caption{How to write algorithms}
%\end{algorithm} 
%
We need a set of directions $v_1, v_2, \ldots v_K$ in the above algorithm. For $d=2$ (that is shapes in the plane) we used $64$ evenly spaced directions. For $d=3$ (that is simplicial complexes in $\R^3$) we used $162$ directions form $S^2$  based on a grid constructed by subdividing an icosohedron. For $0$-dimensional persistence we use the union-find algorithm for efficiency.  The amortized time per operation 
for the union-find algorithm is $O(\alpha(n))$, where $\alpha(n)$ is the inverse of the Ackermann
function $A(n,n)$ which grows extremely quickly. For any reasonable $n$, $\alpha(n)$ is less than $5$. Thus, the amortized running time per operation is effectively a small constant. Thus, in each direction, computing the persistence diagram $X_0(M,v)$ is effectively linear in the number of vertices in the simplicial complex.

For each pair of objects we use the Hungarian algorithm to compute the distances between two persistence diagrams in each of the directions. The runtime complexity of the Hungarian algorithm is  
$O(m^3)$ where $m$ is the number of off-diagonal points in the two diagrams $X$ and $Y$ combined.

\subsection{Results on data sets}

We have used the metric on the space of PHTs to analyze a variety of data sets. In the appendix we consider ellipsoids and hyperboloids with restricted $z$ values. In these examples we have used the algebraic structure to compute and describe what the PHTs are. We also analyze the resulting distance matrices using multidimensional scaling and give a geometrical interpretation of the coordinates. In this section we present the results when applying the PHT to both a shape database (of contractable shapes in the plane) and to a data set of pre-aligned calcanei of primates.

\subsubsection{Results on a shape database} \label{sil}

To examine how well we could measure distances between planar shapes and how well our method can align shapes with respect to scaling, translation,
and rotation we studied the performance on a standard shape data set. A shape database that has been commonly used in image retrieval is the 
MPEG-7 shape silhouette database \cite{Sikora01}. We used a subset of this database \cite{Lateckietal00} which includes seven class of objects:
 Bone, Heart, Glass, Fountain, Key, Fork, and Axe. There were twenty examples for each class for a total of 1400 shapes. The shapes are
 displayed in Figure \ref{images}. 
 {
 \begin{figure}[hbt]
\begin{center}
\includegraphics[height=1.1in]{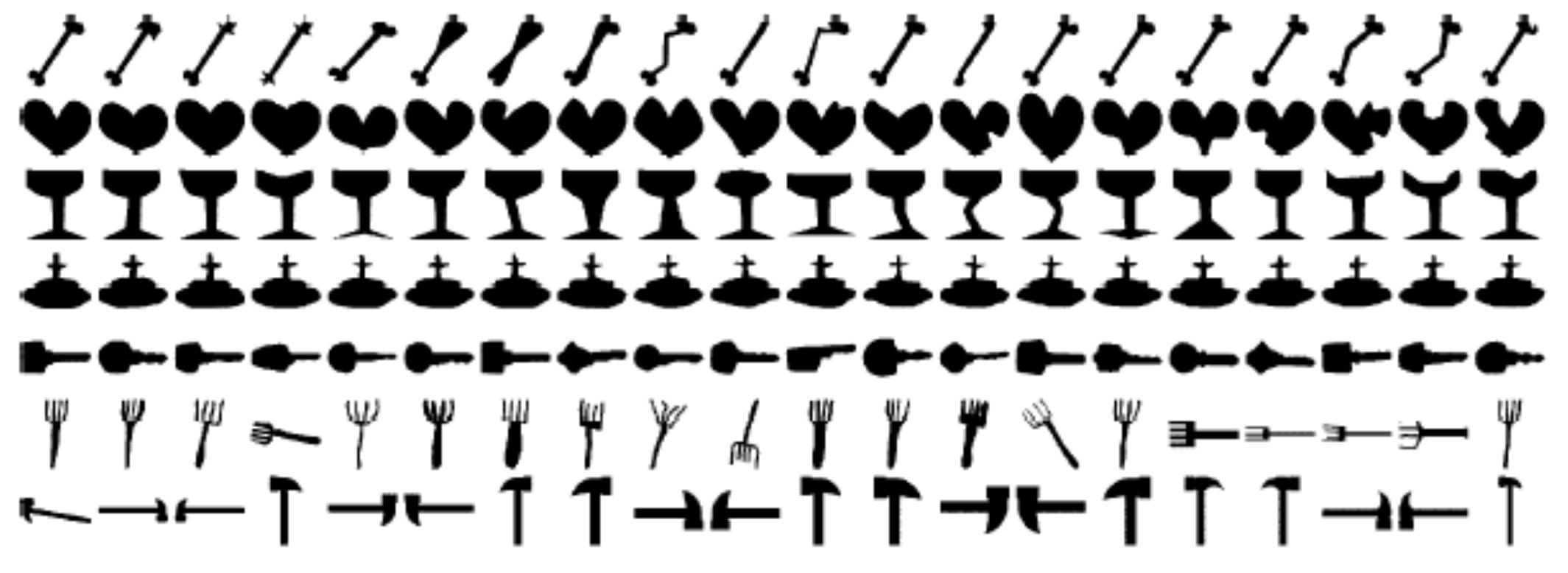}\\
\caption{\label{images}
The subset of the silhouette database. Each row corresponds to one the objects: Bone, Heart, Glass, Fountain, Key, Fork, and Axe. Note that
while the objects are distinct, there is a great deal of variation within each object.
}
\end{center}
\end{figure}
 
 \samepage
 
\begin{figure}[hbt]
\begin{center}a)
\includegraphics[height=2.2in]{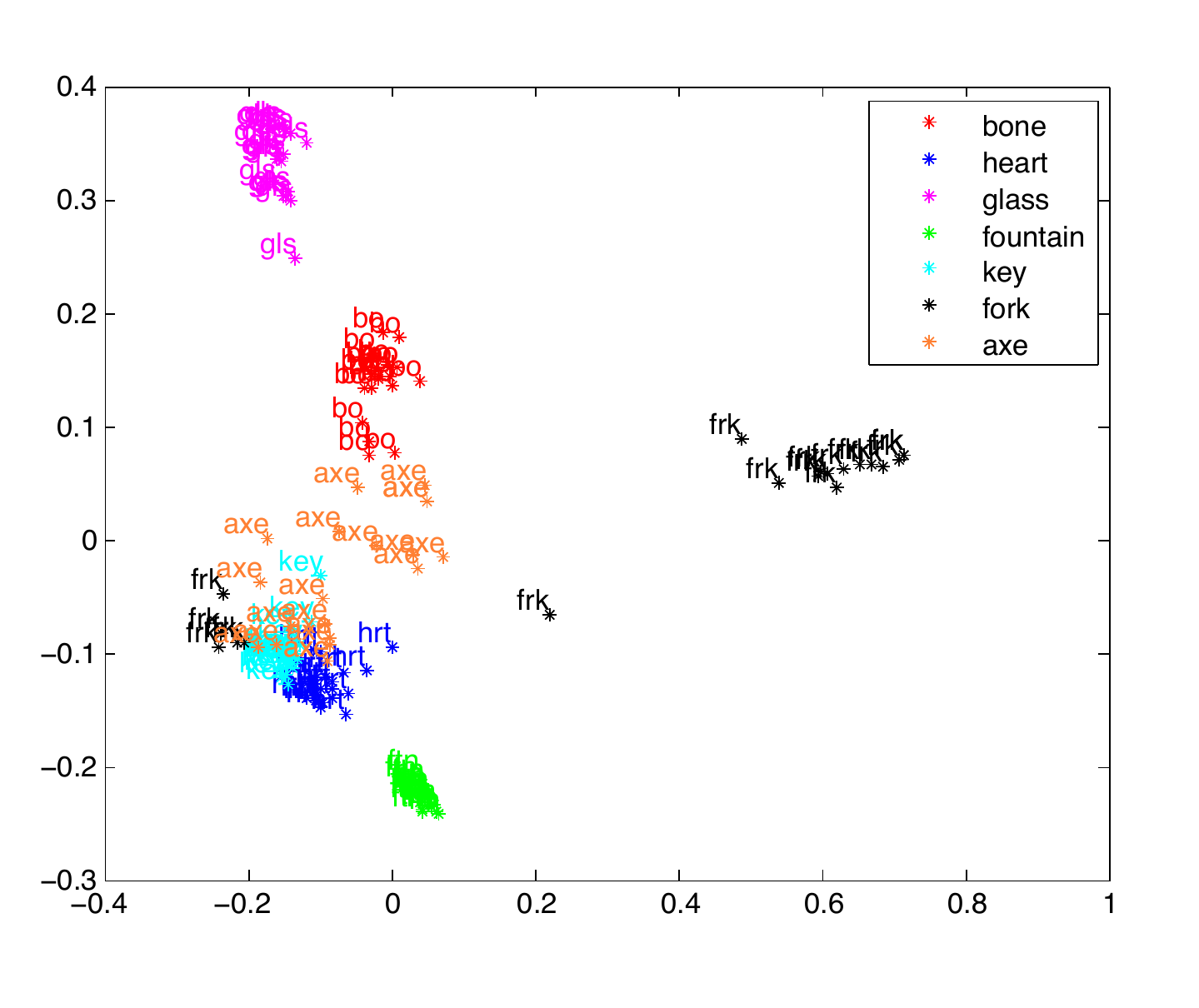} 
b)\includegraphics[height=2.1in]{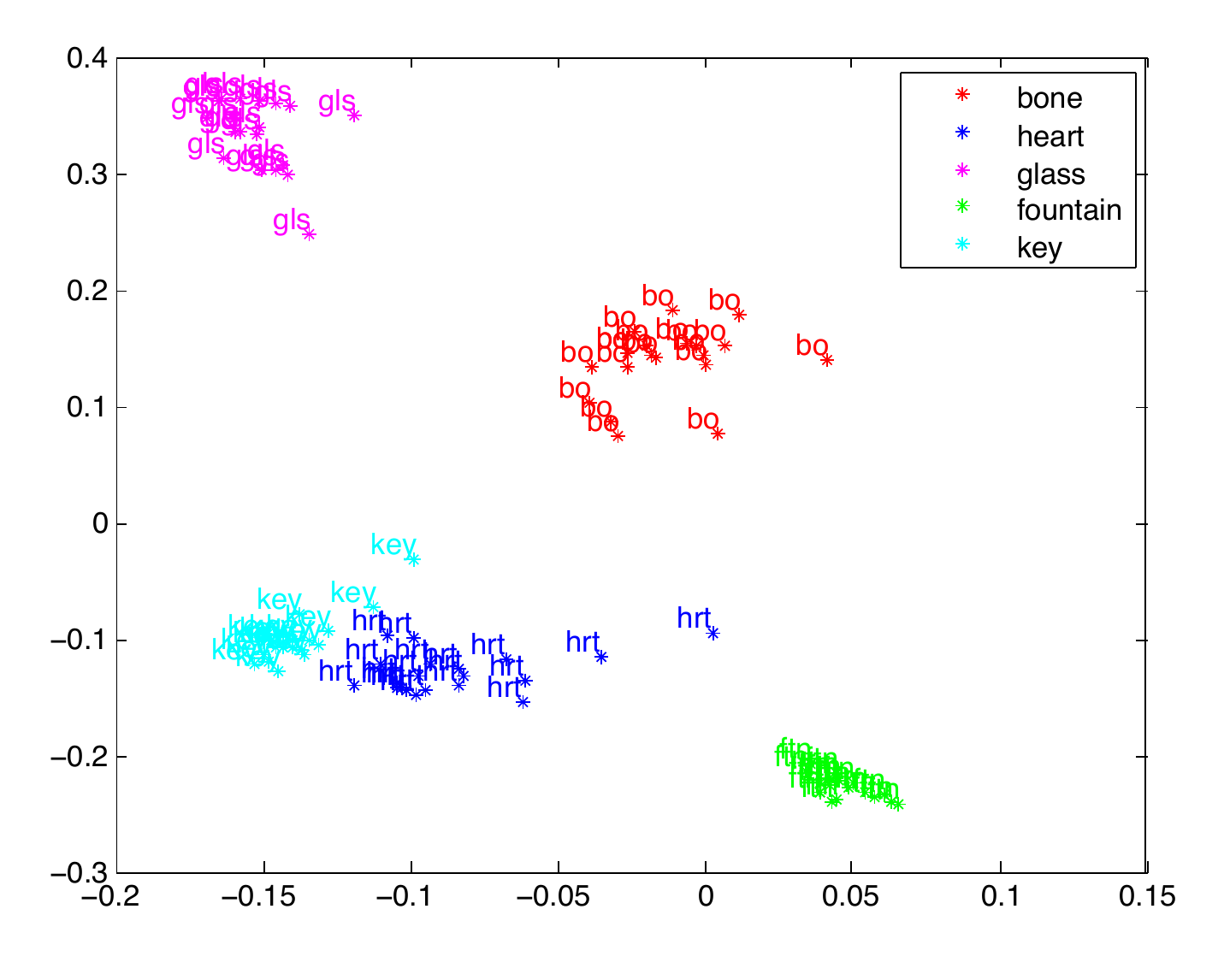} \\
c)\includegraphics[height=5in]{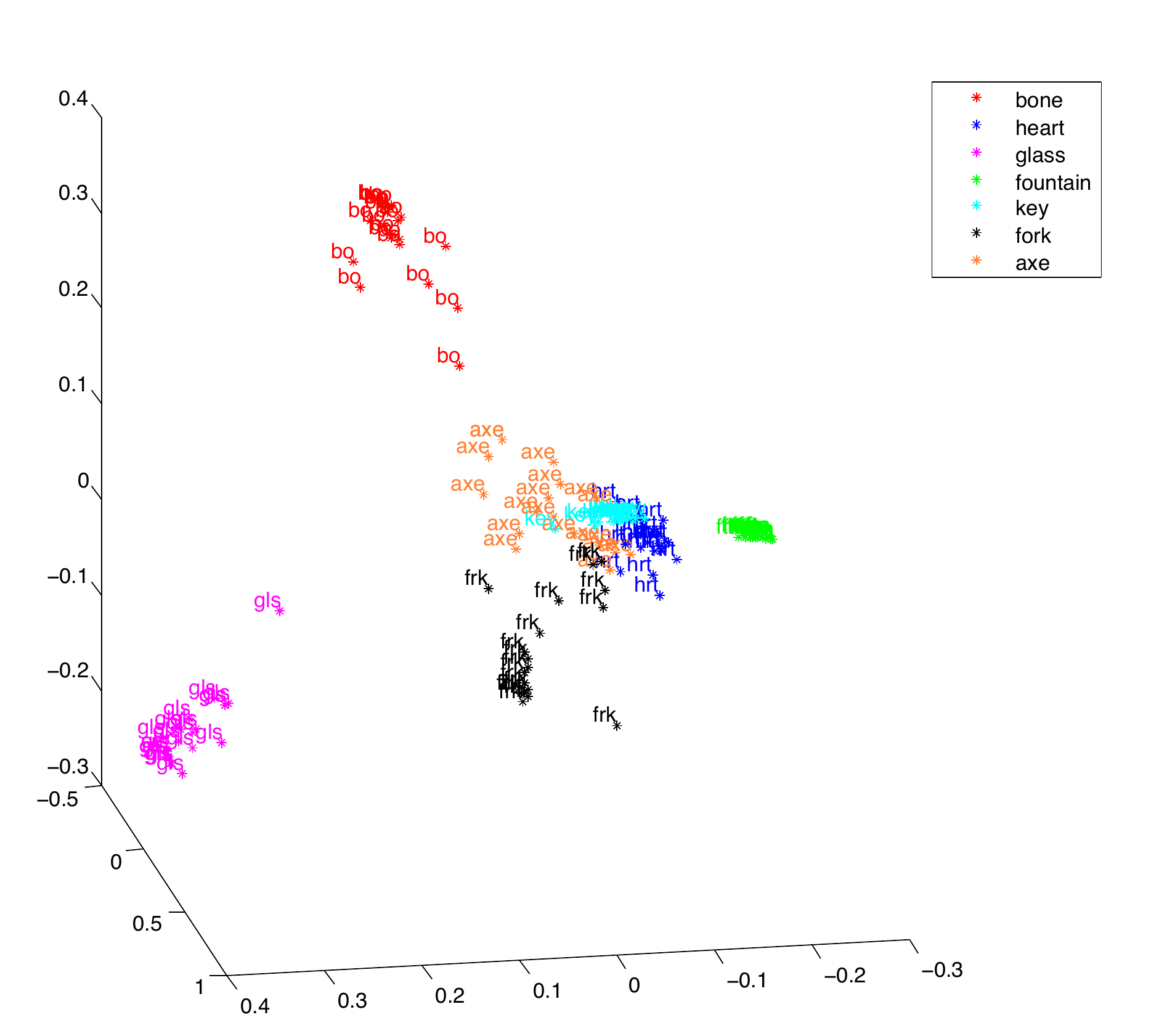}
\caption{\label{imredpics}
a) The two dimensional projection using multidimensional scaling. The classes are Bone in red, Heart in blue, Glass in magenta,
Fountain in green, Key in cyan, Fork in black, and Axe in orange. Note that Axe and Fork do not form tight clusters. b) The two dimensional projection
without including the Fork and Axe class illustrates that all the other classes are clustered. c)  The three dimensional projection using multidimensional scaling. }
\end{center}
\end{figure}
 }

 We used the perimeters of the silhouettes which are available at \cite{Gao}. We applied the alignment algorithm we stated in
Section \ref{align} to shift and scale the silhouettes. These perimeters are all homeomorphic to a circle so we used the $0$-th dimensional 
persistent homology transform with $64$ evenly spaced directions. We computed the distances between all objects - pairwisely checking under different rotations and taking the minimum as outlined in Section \ref{align}. We then used 
multidimensional scaling \cite{KruskalWish78} on the computed distance matrix to project the data into two or three dimensions. In 
Figure \ref{imredpics} we see that except for the Axe and Fork classes the objects are separated.

%(Shape data for the MPEG-7 core experiment CE-Shape-1. http://www.cis.temple.edu/~latecki/TestData/mpeg7shapeB.tar.gz)
% $(http://visionlab.uta.edu/shape_data.htm).$ 

\subsubsection{Results on real data--calcanei of primates}

Information on the pattern of change in anatomical form and diversity of form through time comprises evidence fundamental to hypotheses in evolutionary biology. Often there is great interest in relating the genetic variation in species with variation in phenotypes such as bones. A challenge in modeling phenotypic variation is developing automated methods to measure the variation or distance between shapes  \cite{Boyeretal11,Boyeretal13,Gladmanetal13}.

The data consist of heel bones (calcanei) form 106 extant and extinct primates \cite{Boyeretal13,Gladmanetal13}. The bones were scanned using microCT scanning and the data for each bone consists of thousands of points in $\R^3$. Details can be found in  \cite{BoyerSeiffert13}.  See Figure 
 \ref{calc1} for two pictures of a calcaneus at different angles. See the Appendix for a list. From these distances we constructed a multidimensional scaling plot of the samples with $D=2$ (see Figure \ref{bonesproj}). We compared the pairwise distances between the 
 projections of the bones by three method:
 \begin{enumerate}
 \item[i.] Manual: The original analysis of this sample in Gladman et al. \cite{Gladmanetal13} was based on 27 manually placed landmarks per bone. 
The landmark coordinates were then scaled and aligned by a generalized Procrustes superimposition and finally analyzed with principal coordinates analysis. We then projected the samples onto the first two principle components.
\item[ii.] Automated protocol: An automated method was developed in \cite{Puente13} to compute distances between bones as well as to align the bones to standard orientation. This alignment protocol also outputs pairwise procrustes distances based on 1000 automatically positioned pseudolandmarks.  These pairwise distances can then be projected onto two principle components.
\item[iii.] PHT: We first aligned the bones using the same procedure as in the automated protocol above. We then computed pairwise distances between bones using the PHT. These pairwise distances can then be projected onto two principle components.
\end{enumerate} 
 
\begin{figure}[htb]
\begin{center}
\includegraphics[height=2.5in]{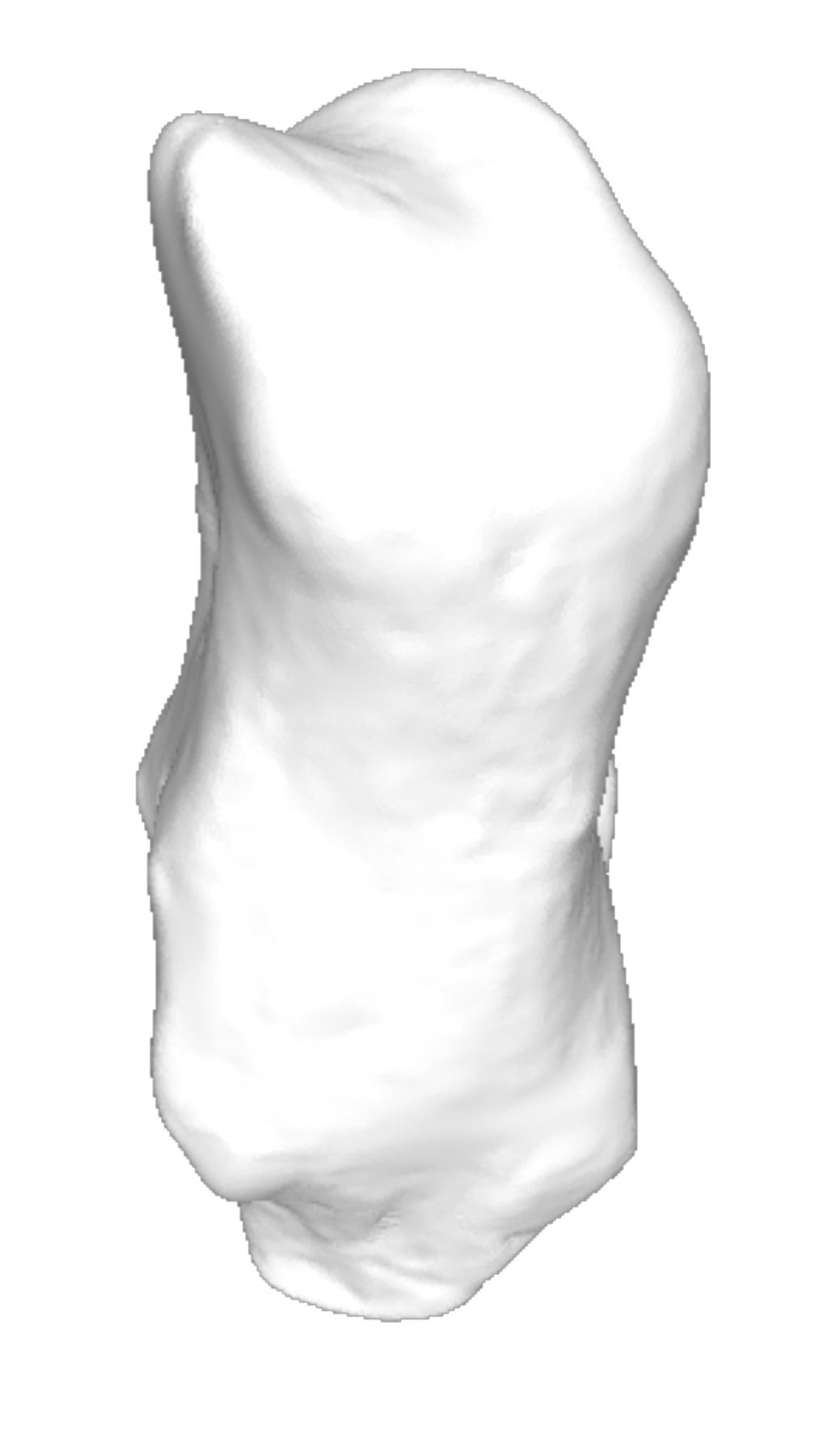}  \hspace{.3in}
\includegraphics[height=2.5in]{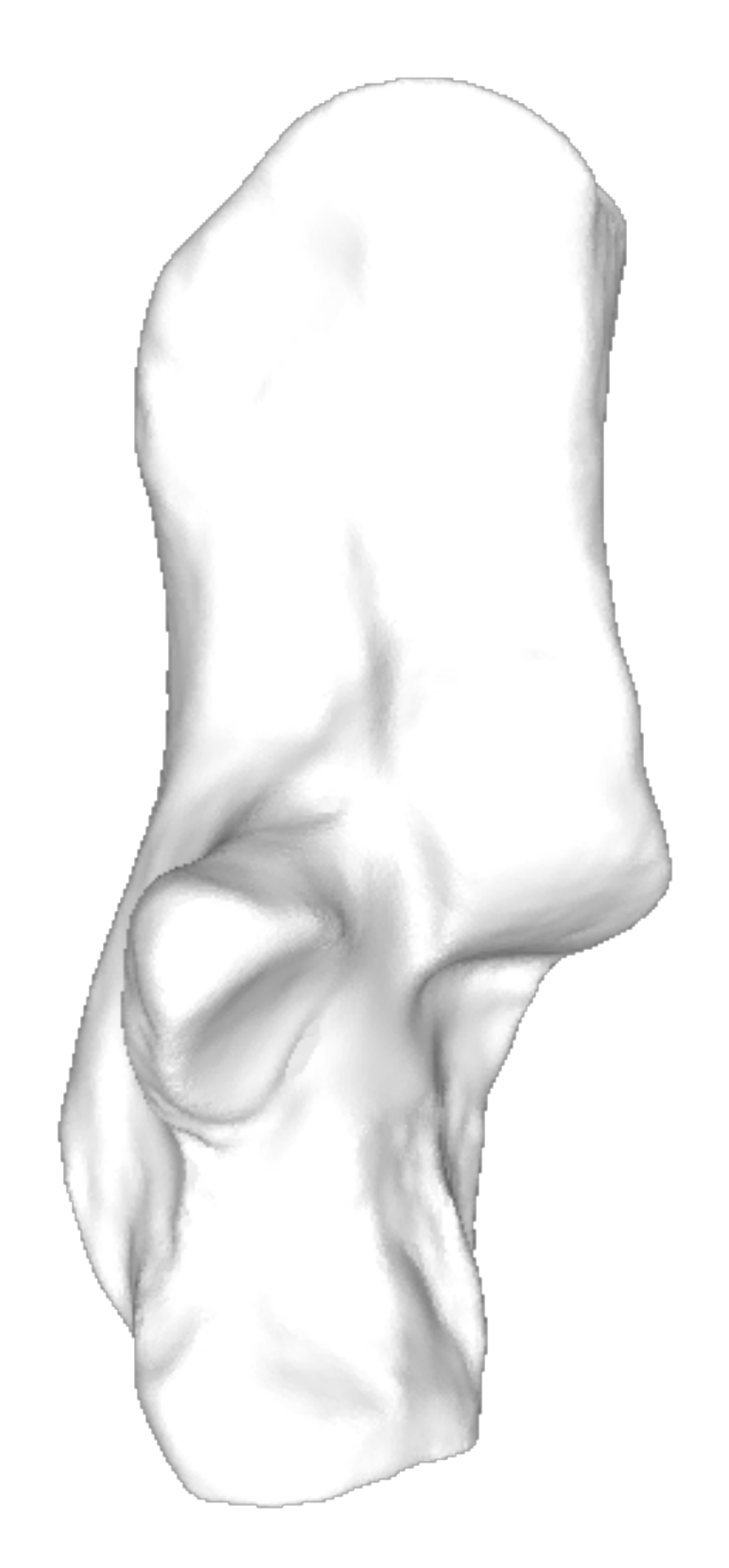}
\caption{\label{calc1}
Images of a calcaneus from two different angles.
}
 \end{center}
\end{figure}
 
 Given these three analyses of the same 106 bones projected into two dimensions. We compared the distances between corresponding bones for each projection by optimizing for rotation, and translation. This was done using the iterative closest point (ICP) algorithm \cite{ICP} which takes two point sets and computes the optimal rotation and translation to align the point sets. We compared the combined Euclidean distance between the aligned points between each distance computation method: \\
 (1) Distance between Automated protocol and PHT: .014 \\
 (2) Distance between Manual and Automated protocol: .015 \\
 (3) Distance between Manual and PHT: .016.\\
  The two automated methods seem to be closer to each other than the manual landmark based method. It should be recalled here that we used the same alignment processes for both the Automated and the PHT -- if we had used a different alignment process, such as that in Section \ref{align}, we would have subtlety affected the distances. In addition the distance between the PHT transform and the manual protocol is greater than the distance between the automated protocol and the manual protocol.  Results of MDS on all three distances are displayed in  Figure \ref{bonesproj}. A qualitative analysis suggests that the PHT distances may be outperforming the other two methods.

%\begin{table}[ht]
%\begin{center}
%\caption{\label{sumcomp} Comparison of linear correlations different shape metrics for 106 calcanei sample.  Upper right  -- $P$-value of no correlation / lower left - linear correlation coefficient $(r)$.}{
%\centering
%\begin{tabular}{c c c c}
%\hline
%D1 & Man & Puente & PHT \\
%\hline
%Man & --	 & $\ll 0.0001$	& $\ll 0.0001$ \\
%Puente &	$-0.96$	& --	&$\ll 0.0001$ \\
%PHT	 & $0.93$ & $-0.94$ & -- \\
%& & & \\
%\hline
%D2 & Man & Puente & PHT \\
%\hline
%Man & --	 & $\ll 0.0001$	& $\ll 0.0001$ \\
%Puente &	$-0.50$	& --	&$\ll 0.0001$ \\
%PHT	 & $0.60$ & $-0.83$ & -- \\
%\end{tabular}}
%\end{center}
%\end{table}

{\begin{figure}[hbt]
\begin{center}
\includegraphics[height=2.3in]{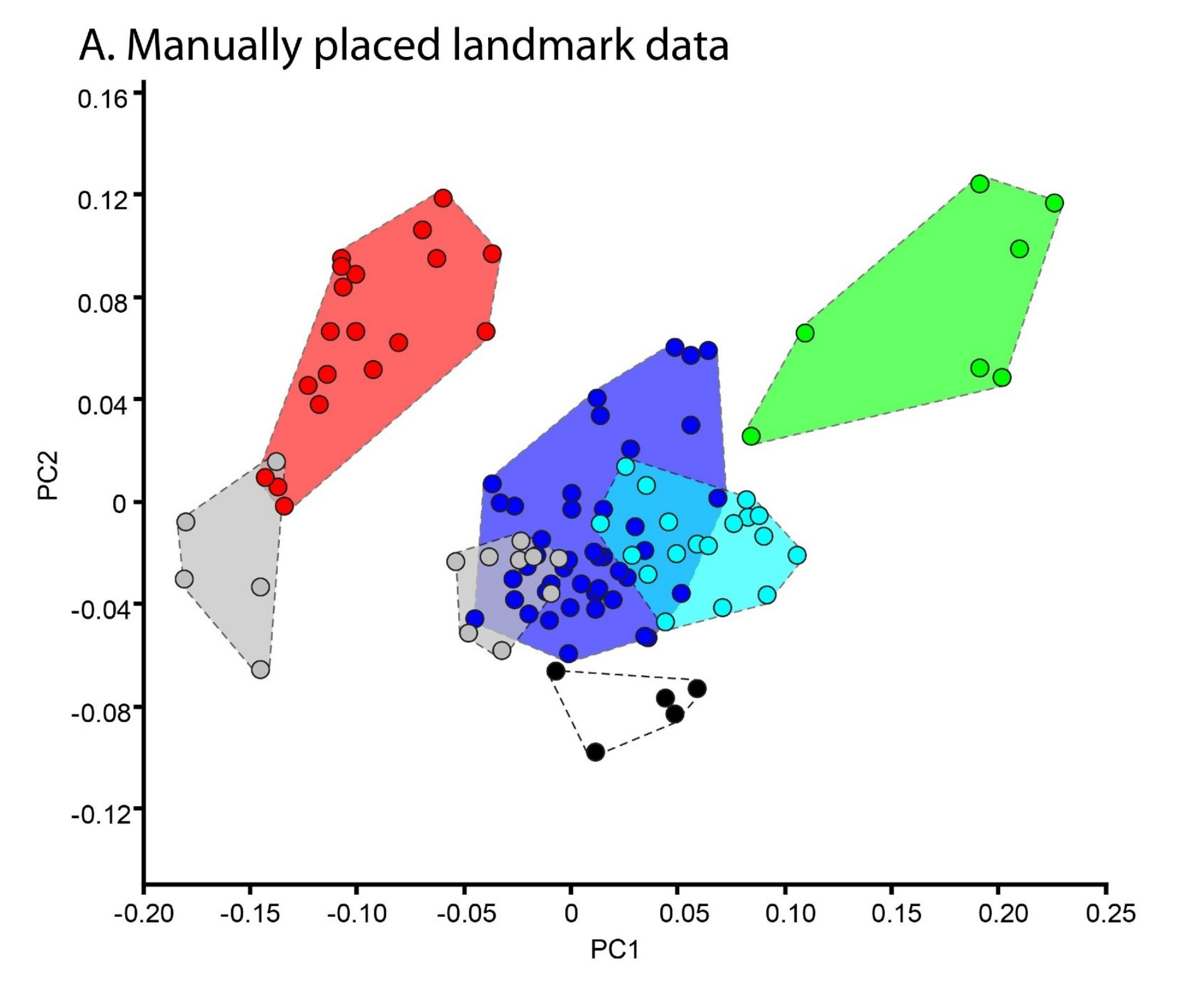} \includegraphics[height=2.3in]{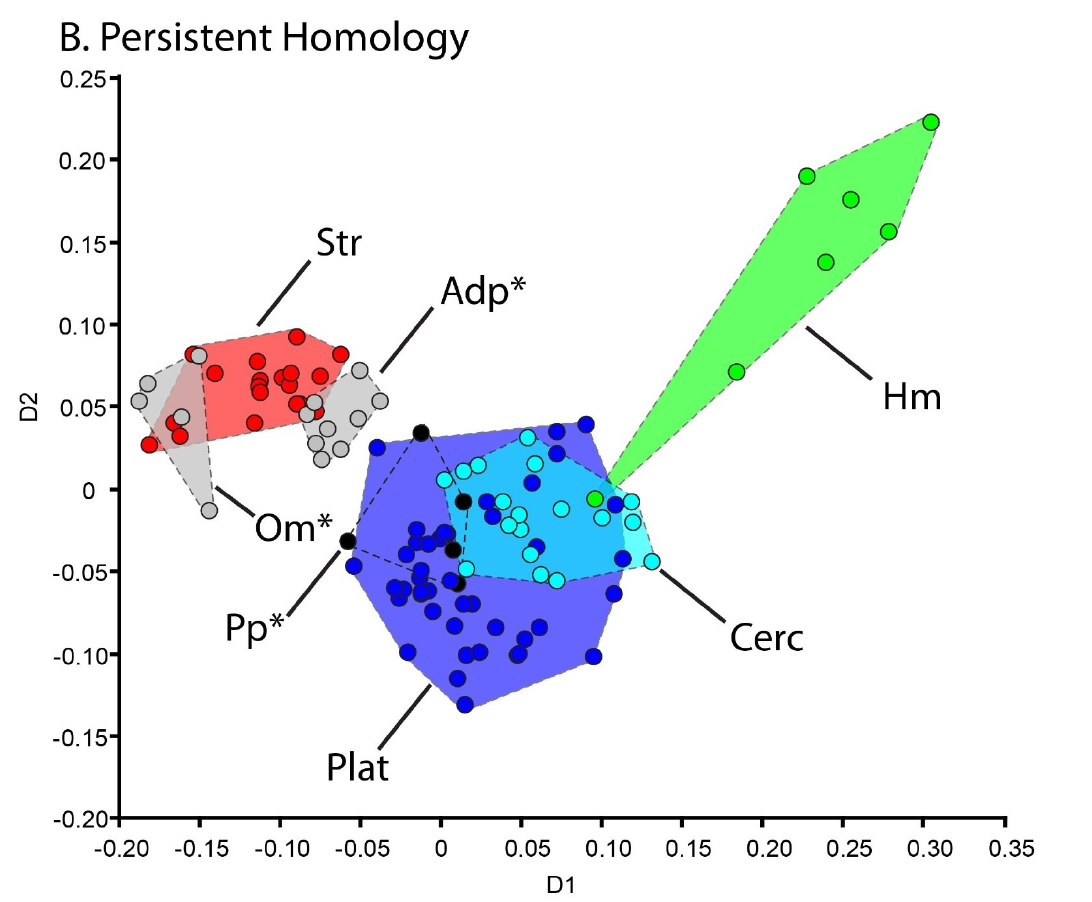} \\
\includegraphics[height=2.3in]{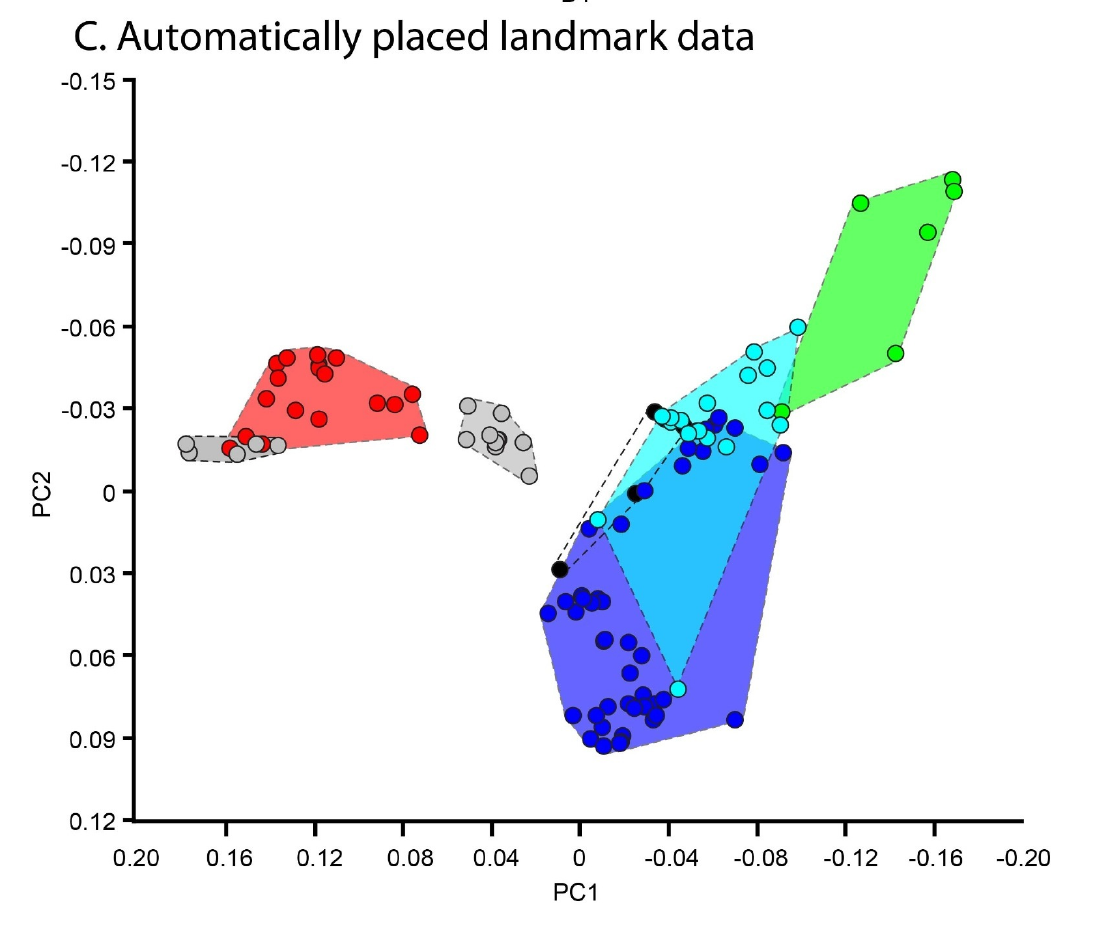}
\caption{\label{bonesproj} Phenetic clustering of phylogenetic groups of primate calcanei $(n=106)$.  67 genera are represented. See appendix for specimen information. Asterisks indicate groups of extinct taxa. Abbreviations: Str, Strepsirrhines; Plat, platyrrhines; Cerc, Cercopithecoids; Om, Omomyiforms; Adp, Adapiforms; Pp, parapithecids; Hmn, Hominoids.  Note that more primitive ÒprosimianÓ taxa cluster separately from ÒsimiansÓ (Om, Adp, Str.).  Also note that ÒmonkeysÓ (Plat, Cerc, Pp) cluster mainly separately from apes (Hmn). This is virtually the same type of clustering observed with more labor intensive (and fundamentally subjective) manually collected landmarks.  The position of  adapiforms between strepsirrhines and platyrrhines is consistent with previous descriptions of the phenetic affinities of this group. (a) is using the manual alignment method, (b) is using the PHT method, (c) is the automated landmark based method.
}
 \end{center}
\end{figure}

\section{Discussion}

In this paper we stated the Persistent Homology Transform as a statistic to capture the information in a shape. Our main result is to prove that this
statistic is sufficient. Two useful features of our method is that we do not require user-specified landmarks, we also can measure distances between
shapes that are not isomorphic since we are using topological features.

Several questions remain regarding this approach including:
\begin{enumerate}
\item[(1)] We suspect that our sufficiency results will extend to simplicial complexes lying in Euclidean spaces for dimensions greater than $3$ and to more general compact subsets of Euclidean space such as manifolds. However, we do not have a proof;
\item[(3)] We are very interested in developing more robust methods for aligning objects. The method for accounting for rotation in Section \ref{align} is too computationally heavy for surfaces in $\R^3$, but reasonable for shapes in $\R^2$;
\item[(4)] Can we examine the sufficiency of other geometric and topological summaries of the data and use this to better understand classic shape space models.
\end{enumerate}

\section*{Funding}
This work was supported by the National Institutes of Health (Systems Biology): [5P50-GM081883, AFOSR: FA9550-10-1-0436, and NSF CCF-1049290 all to SM].

\section*{Acknowledgements}
KT would like to thank Shmuel Weinberger for helpful conversations.
SM would like to thank Jesus Puente, Ingrid  Daubechies, and Jenny Tung for useful comments. SM would also like to thank W. Mio and L. Ma for Figure 10.

\appendix
\section{Calcaneal data set}
Full calcaneal data set. Bone \# column can be used to look up
specimens in the 3D alignment file.
\begin{table}[ht]
\centering
\begin{tabular}{|c|c|c|}
\hline
Taxon & Specimen ID & Bone \# \\
\hline
\hline 
{\em Avahi laniger} & AMNH 170461 & 1 \\
\hline
{\em Cheirogaleus major} & AMNH 100640 & 2 \\
\hline
{\em Daubentonia madagascariensis} &  AMNH 185643 & 3\\
\hline
{\em Eulemur fulvus} & AMNH 17403 & 4 \\
\hline
{\em Eulemur fulvus} &AMNH 31254 &  5 \\
\hline
{\em Hapalemur griseus} &  AMNH 170675 & 6 \\
\hline
{\em Hapalemur griseus} &  AMNH 170689 & 7\\
\hline
{\em Hapalemur griseus} & AMNH 61589 & 8\\
\hline
{\em Indri indri} &  AMNH 100504 & 9 \\
\hline
{\em Indri indri} & AMNH 208992 & 10 \\
\hline
{\em Lemur catta} &  AMNH 150039 & 11 \\
\hline
{\em Lemur catta} & AMNH 170739 & 12 \\
\hline
{\em Lemur catta} & AMNH 22912 & 13 \\
\hline
{\em Lepilemur mustelinus} &  AMNH 170565 & 14 \\
\hline
{\em Lepilemur mustelinus} &  AMNH 170568 & 15 \\
\hline
{\em Lepilemur mustelinus} &  AMNH 170569 & 16 \\
\hline
{\em Propithecus verreauxi}& AMNH 170463 & 17 \\
\hline
{\em Propithecus verreauxi} & AMNH 170491 & 18 \\
\hline
{\em Varecia variegata} &  AMNH 100512 & 19 \\
\hline
{\em Alouatta seniculus} & AMNH 42316 & 20 \\
\hline
{\em Alouatta seniculus} &  SBU NAl13 & 21 \\
\hline
{\em Alouatta sp.} & SBU NAl17 & 22 \\
\hline
{\em Alouatta sp.} &  SBU NAl18 &  23 \\
\hline
{\em Aotus azarae} &  AMNH 211482 & 24 \\
\hline
{\em Aotus infulatus} & AMNH 94992 & 25 \\
\hline
{\em Aotus sp.} & AMNH 201647 & 26 \\
\hline
{\em Ateles paniscus} &  SBU NAt10 & 27 \\
\hline
{\em Ateles sp.} &  SBU NAt13 & 28 \\
\hline
{\em Ateles sp.} & SBU NAt18 & 29 \\
\hline
{\em Brachyteles arachnoides} & AMNH 260 & 30 \\
\hline
{\em Cacajao calvus} &  AMNH 70192 & 31 \\
\hline
{\em Cacajao calvus} & SBU NCj1 &32 \\
\hline
{\em Callicebus donacophilus} & AMNH 211490 & 33 \\
\hline
{\em Callicebus moloch} & AMNH 244363 & 34 \\
\hline
{\em Callicebus moloch} & AMNH 94977 & 35 \\
\hline
{\em Callimico goeldi}& AMNH 183289 & 36 \\
\hline
{\em Callimico goeldi} & SBU NCa1 & 37 \\
\hline
{\em Callithrix jacchus} & AMNH 133692 &38 \\
\hline
{\em Callithrix jacchus} &  AMNH 133698 & 39 \\
\hline
{\em Cebuella pygmaea} &  AMNH 244101 & 40 \\
\hline
{\em Cebuella pygmaea} & SBU NC1 & 41 \\
\hline
{\em Cebus apella} & SBU NCb4 & 42 \\
\hline
{\em Cebus sp.} & SBU NCb5 & 43 \\
\hline
{\em Chiropotes satanus} &  AMNH 95760 &  44 \\
\hline
{\em Chiropotes satanus} &  AMNH 96123 &  45 \\
\hline
{\em Chiropotes sp.} & SBU NCh2 &46 \\
\hline
{\em  Leontopithecus rosalia}  & AMNH 137270 & 47\\
\hline
{\em Leontopithecus rosalia } &  AMNH 60647 & 48 \\
\hline
{\em Pithecia monachus} &  AMNH 187978 & 49 \\
\hline
{\em Pithecia pithecia} &  AMNH 149149 & 50 \\
\hline
{\em Saguinus midas} &  AMNH 266481 & 51 \\
\hline
{\em Saguinus mystax } & AMNH 188177 & 52 \\
\hline
{\em Saguinus sp.} & SBU NSg12 & 53 \\
\hline
{\em Saguinus sp.} &  SBU NSg2 & 54 \\
\hline
{\em Saimiri boliviensis} & AMNH209934 & 55\\
\hline
{\em Saimiri boliviensis} & AMNH211650 & 56 \\
\hline
{\em Saimiri boliviensis} & AMNH211651 & 57 \\
\hline
\hline
\end{tabular}
\end{table}

\newpage

\begin{table}[ht]
\centering
\begin{tabular}{|c|c|c|}
\hline
Taxon & Specimen ID & Bone \# \\
\hline
\hline 
{\em Saimiri sciureus} & AMNH188080 & 58 \\
\hline
{\em Saimiri sp.} & SBU NSm2 & 59 \\
\hline
{\em Cercopithecus sp.} &  SBU No Number & 60\\
\hline
{\em Cercopithecus sp.} & SBU No Number & 61 \\
\hline
{\em Chlorocebus aethiops}& SBU OCr7 & 62 \\
\hline
{\em Chlorocebus cynosuros}  &  AMNH 80787 & 63 \\
\hline
{\em Colobus geureza} & AMNH 27711 & 64 \\
\hline
{\em Erythrocebus patas} & AMNH 34709 & 65 \\
\hline
{\em Lophocebus albigena} & AMNH 52603 & 66\\
\hline
{\em Macaca nigra}&  SBU OCn1 & 67 \\
\hline
{\em Macaca tonkeana} & AMNH 153402 & 68 \\
\hline
{\em Mandrillus sphinx} & AMNH 89367 & 69\\
\hline
{\em Nasalis larvatus} &  AMNH 106272 & 70 \\
\hline
{\em Papio hamadryas} & AMNH 80774 & 71 \\
\hline
{\em Piliocolobus badius} &  AMNH 52303 &  72 \\
\hline
{\em  Piliocolobus badius} &  ED 4651 & 73 \\
\hline
{\em Pygathrix nemaeus} & AMNH 87255 & 74 \\
\hline
{\em Theropitheucs gelada} & AMNH 201008 & 75 \\
\hline
{\em Trachypithecus obscurus} &  AMNH 112977 & 76 \\
\hline
{\em Gorilla sp.} & AD 6001 & 77 \\
\hline
{\em Hylobates lar} & AMNH 119601 & 78 \\
\hline
{\em Pan troglodytes} & AMNH 51202 & 79 \\
\hline
{\em Pan troglodytes} & AMNH 51278 & 80 \\
\hline
{\em Pongo pygmaeus} & AMNH 28253 & 81 \\
\hline
{\em Symphalangus syndactylus} & AMNH 106583 & 82\\
\hline
{\em Cantius abditus} & USGS 6783 & 83 \\
\hline
{\em Cantius sp.} &  USGS 6774 & 84 \\
\hline
{\em Cantius trigonodus} & AMNH 16852 & 85 \\
\hline
{\em Cantius trigonodus} & USGS 21829& 86 \\
\hline
{\em Cebupithecia sarmientoi} & UCMP 38762* & 87 \\
\hline
{\em Marcgodinotius indicus} & GU 709 & 88 \\
\hline
{\em Mesopithecus pentelici} & MNHN PIK-266* & 89 \\
\hline
{\em Neosaimiri fieldsi} & IGM-KU 89202* & 90 \\
\hline
{\em Neosaimiri fieldsi} & IGM-KU 89203* & 91 \\
\hline
{\em Notharctus sp.} & AMNH 55061 & 92 \\
\hline
{\em Notharctus tenebrosus} & AMNH 11474 & 93 \\
\hline
{\em Omomyid} & AMNH 29164 & 94 \\
\hline
{\em Omomys sp.} &  UM 98604 &95 \\
\hline
{\em Oreopithecus bambolii} & NMB 37* & 96 \\
\hline
{\em Ourayia uintensis} & SDNM 60933 & 97\\
\hline
{\em Parapithecid} &  DPC 15679 & 98\\
\hline
{\em Parapithecid} &  DPC 20576 & 99 \\
\hline
{\em Parapithecid}& DPC 2381 & 100 \\
\hline
{\em Parapithecid} &  DPC 8810 & 101 \\
\hline
{\em Proteopithecus sylviae} &  DPC 23662A & 102 \\
\hline
{\em Smilodectes gracilis} &  AMNH131763 & 103 \\
\hline
{\em Smilodectes gracilis} & AMNH131774 & 104 \\
\hline
{\em Teihardina belgica} & IRSNB 16786-03 & 105 \\
\hline
{\em Washakius insignis} & AMNH 88824 & 106 \\
\hline
\hline
\end{tabular}
\end{table}

\clearpage

\section{Examples of PHT of families of surfaces}
The purpose of this appendix is to examine in detail the process of the PHT of some parameterized families of shapes and also to consider the distances between the PHTs of these shapes. This should help the reader gain an intuition about the PHT. We will consider quadric ellipsoids and hyperboloids which have had restricted $z$-values. We will explain what the persistence diagrams in each direction are. In the first case (ellipsoids) we will consider the normalization process. We will calculate the distances between the PHTs for sets within these families and analyze the results using multidimensional scaling. We also, in the case of ellipsoids, compare the values of the distances found by the algorithm \ref{subsec:alg} of a surface mesh to the value using the exact algebraic set and computed by numerical integration . Showing that these values are close reassures that we are not losing too much information by taking finite approximations (both in the surface mesh stage and the finitely many directions instead of an integral stage). The code used to compute the PHTs  of the surface meshes, and the distances between them, within these examples is available on the journal's website.

\subsection[A]{PHT of ellipsoids}
Let $E(a,b,c)$ denote the ellipsoid with Cartesian equation
$$\frac{x^2}{a^2}+\frac{y^2}{b^2}+\frac{z^2}{c^2}=1.$$
The algebraic description of the ellipsoids means that it is possible for us to find formulas to describe the persistent homology diagrams for the height functions in each direction. Fix a direction described by the unit vector $v=(v_1,v_2,v_3).$

The $X_1( E(a,b,c), v)$ in the PHT of $E(a,b,c)$ has no off diagonal points for all $v\in S^2$. The $X_0( E(a,b,c), v)$ and $X_2( E(a,b,c), v)$ each have exactly one off diagonal point. These correspond to the essential classes of the ellipsoid. The $H_0$ class appearing  when the ellipsoidis first contacted and the $H_2$ class appearing when the entire ellipsoid is completed. Thus to compute these diagrams it is enough to compute the minimum and maximum value of 
$$f(x,y,z):=xv_1 + yv_2 + zv_3$$ for $(x,y,z) \in E(a,b,c)$, that is $(x,y,z)$ that satisfy $$g(x,y,z):=\frac{x^2}{a^2}+\frac{y^2}{b^2}+\frac{z^2}{c^2}=1.$$ Geometrically this occurs when the normal to the surface is $\pm v$.

%
%To compute these minimum and maximum values we can use Lagrange multipliers. The $(x,y,z)$ at which $f$ attains a maximum or minimum (subject to the constant $g(x,y,z)=1$) must satisfy $\Delta f = \lambda \Delta g$ for some $\lambda \in \mathbb{R}$. This gives the set of equations:
%\begin{align*}
%x=\frac{v_1a^2}{2\lambda}& &y=\frac{v_2b^2}{2\lambda}& &z=\frac{v_3c^2}{2\lambda}.
%\end{align*}
%Combining these with $g(x,y,z)=1$ we conclude $4\lambda^2=a^2v_1^2 + b^2v_2^2 + c^2v_3^2$. Considering the geometry we infer that 
%$$f(x,y,z)= \frac{v_1a^2}{2\lambda}v_1 +\frac{v_2b^2}{2\lambda}v_2+\frac{v_3c^2}{2\lambda}v_3$$ has a maximum and minimum at $\pm \sqrt{a^2v_1^2 + b^2v_2^2 + c^2v_3^2}$. 
Using the method of Lagrange multipliers, the PHT of $E(a,b,c)$ is calculated to be\footnote{Note here that we are recording the persistence diagrams by listing the off diagonal points. They also contain countably infinite copies of the diagonal.} 
\begin{align*}
X_0(E(a,b,c), (v_1,v_2, v_3)) &= \left\{(-\sqrt{a^2v_1^2 + b^2v_2^2 + c^2v_3^2},\infty)\right\},\\
X_1(E(a,b,c), (v_1,v_2, v_3)) &= \left\{\right\},\\
X_2(E(a,b,c), (v_1,v_2, v_3)) &= \left\{(\sqrt{a^2v_1^2 + b^2v_2^2 + c^2v_3^2},\infty)\right\},\\
%$$\colvec{v_1}{v_2}{v_3}\mapsto \left( \left\{(-\sqrt{a^2v_1^2 + b^2v_2^2 + c^2v_3^2},\infty)\right\}, \emptyset , \left\{(\sqrt{a^2v_1^2 + b^2v_2^2 + c^2v_3^2},\infty)\right\}\right).$$
\end{align*}

We may or may not wish to normalize the ellipses. By symmetry these ellipsoids are already centered. If we were to normalize them with respect to size by the process described in section \ref{align} we would need to calculate, for each $E(a,b,c)$, the scaling factor
\begin{align} \label{eq:size}
I(a,b,c)=\frac{1}{4\pi}\int_{\beta=-\pi}^\pi \int_{\alpha=-\pi/2}^{\pi/2}  \sqrt{c^2\sin^2 \alpha + \cos^2\alpha  (a^2 \cos^2 \beta + b^2\sin^2 \beta)}\cos \alpha \, d\alpha \, d\beta.A
\end{align} 
All our integrals over the sphere are done in the polar coordinates $(\alpha, \beta) \mapsto (\cos \alpha \cos \beta, \cos \alpha \sin \beta, \sin \alpha)$. The Jacobian for this parameterization is $\cos\alpha$. Unfortunately the integral in \eqref{eq:size} does not have a nice closed form so we must use numerical integration to compute it. Given a triple $(a,b,c)$ we can rescale $E(a,b,c)$ to find a nomalized (with respect to size) ellipsoid 
$$\widehat{E}(a,b,c) = E\left(\frac{a}{I(a,b,c)},\frac{b}{I(a,b,c)}, \frac{c}{I(a,b,c)}\right).$$

To find the distances between the centered and normalized ellipses we would then need to consider the pairwise
the different distances under all alignments and take the minimal one.

%%IS this a\leq b\leq c????=-0 
We looked at a set of ellipsoids $$\{E(a,b,c): a,b,c \in \{1,2,3,4,5\}\},$$ with their size and location in space fixed. We computed all the pairwise distances  and then performed multi-dimensional scaling to the matrix of distances squared that resulted. After applying multidimensional scaling there were natural geometric interpretations of the coordinates. Furthermore, since ellipsoids are such nice algebraically defined sets it is possible to have formulae for the persistence diagrams. This allows us to find the true distance between the persistent homology transforms by (numerical) integration. We also computed the distances using the algorithm outlined in \ref{subsec:alg}. By computing the distances these two ways we can see how much error is introduced (in this example) through the  finite approximations of the surface by using a finite surface mesh and through the finite approximation of integral over the sphere by averaging using a finite number of directions. The relative error of the computed distances by the two different methods of computation was bounded by $0.0532$. The relative error was always positive (the algorithm always underestimated the distances compared to the distances computed by the integral) with mean $0.0244$. 

We now will show in detail the computation involved. Let $E_1=E(a_1,b_1,c_1)$ and $E_2=E(a_2, b_2, c_2)$. The symmetry of the ellipsoids implies that 
$$\d(X_0(E_1,v),X_0(E_2,v))=\d(X_2(E_1,v),X_2(E_2,v))$$ for all $v\in S^2$. Also none of the $X_1(E_i,v)$ have any off diagonal points and hence 
$$\int \d(X_1(E_1,v),X_1(E_2,v))\, dv=0.$$ Together these simplify the calculation of $\d(E_1, E_2)$.
\begin{align*}
\d&(E_1, E_2)\\
=& \frac{2}{\text{vol}(S^2)} \int_{v\in S^2} \d(X_0(E_1, v), X_0(E_2, v))\, dv\\
=&\frac{2}{4\pi}\int_{\beta=-\pi}^\pi \int_{\alpha=-\pi/2}^{\pi/2} d(X_0(E_1, v(\alpha,\beta)),X_0(E_2, v(\alpha,\beta))) \cos \alpha \, d\alpha\,  d\beta\\
=&\frac{2}{4\pi}\int_{\beta=-\pi}^\pi \int_{\alpha=-\pi/2}^{\pi/2} \bigg| \sqrt{c_1^2\sin^2 \alpha + \cos^2\alpha  (a_1^2 \cos^2 \beta + b_1^2\sin^2 \beta)}\\
&\qquad -\sqrt{c_2^2\sin^2 \alpha + \cos^2\alpha  (a_2^2 \cos^2 \beta + b_2^2\sin^2 \beta)} \bigg|\cos \alpha \, d\alpha \, d\beta.
\end{align*}

For the set of ellipsoids $\{E(a,b,c): a,b,c \in \{1,2,3,4,5\}\}$ we computed the distance matrix by the formula above and performed multidimensional scaling to this distance matrix. This analysis seemed to show that these persistent homology transforms of ellipsoids effectively lie in a 3 dimensional space. 
%This three dimensional approximation had stress\footnote{This is calculated by $\sqrt{(\sum (d_{ij}-\hat{d}_{ij})^2)/(\sum d_{ij}^2}$ where $d_{ij}$ were the true distances calculated by the formulae and the $\hat{d}_{ij}$ were the distances as calculated taking the first three coordinates of the multidimensional scaling.} of $7.26\%$ which is deemed between ``fair'' and ``good'' \cite{Kruskal1964}. 
The first dimension (that corresponding the the highest eigenvalue) is effectively a linear function of the size of the ellipsoid (as measured by $I(a,b,c)$ as calculated in equation \eqref{eq:size}). This is illustrated in Figure \ref{fig:1score}.
\begin{figure}[hbt]
\begin{center}
\includegraphics[height=3in]{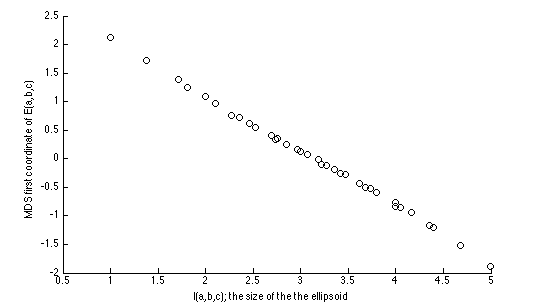}
\caption{The first coordinate of the MDS for $E(a,b,c)$ is effectively a linear function for $I(a,b,c)$ (the size of the ellipsoid).}\label{fig:1score}
\end{center}
\end{figure}

In next two dimensions, the vector of direction is symmetric in relation to the ratios of $a,b$ and $c$. When $a=b=c$, that is we have a sphere, the values of the second and third coordinate are both zero. When the ratios are equal then the points lie in the same direction. In Figure \ref{fig:23scores} we have plotted the second and third coordinates from the multi-dimensional scaling analysis. The eigenvalues corresponding to these two dimensions were equal so the choice of coordinates in this plane were arbitrary.

\begin{figure}[hbt]
\begin{center}
\includegraphics[height=4in]{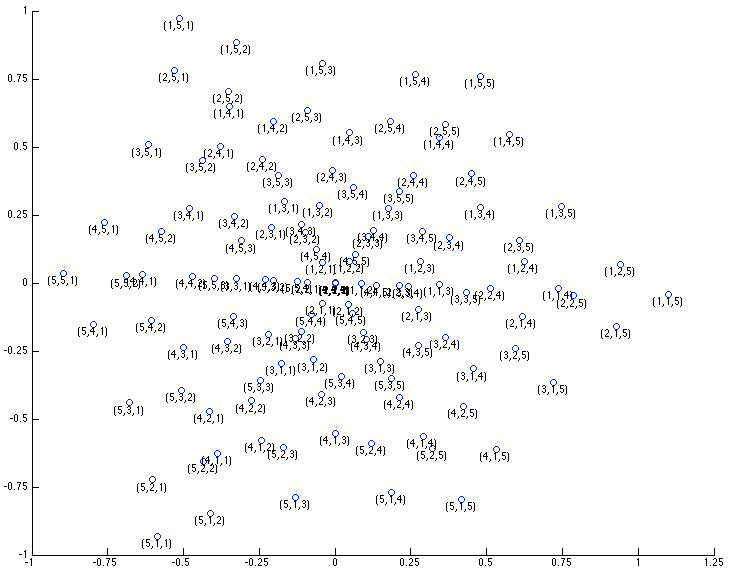}
\caption{The plot of the second and third MDS scores for the set of ellipsoids.}\label{fig:23scores}
\end{center}
\end{figure}

\newpage

\subsection{PHT of hperboloids with restricted $z$-values}
We now will consider a family of hyperboloids with restricted $z$-values -- cut-off hyperboloids. By this we mean sets of shapes of the form
$$Hyp(a,b,c):=\left\{(x,y,z): \frac{x^2}{a^2} + \frac{y^2}{b^2} - \frac{z^2}{c^2}=1 \text{ and } z\in [-1,1]\right\}.$$
Each $Hyp(a,b,c)$ is a surface whose boundary is the pair of ellipses
$\{(x,y,1): \frac{x^2}{a^2} + \frac{y^2}{b^2} =1+ \frac{1}{c^2}\}$ and $\{(x,y,-1): \frac{x^2}{a^2} + \frac{y^2}{b^2} =1+ \frac{1}{c^2}\}.$
An example of such a surface is drawn in Figure \ref{fig:hyperevents}.
\begin{figure}[ht]
\begin{center}
\includegraphics[height=2in]{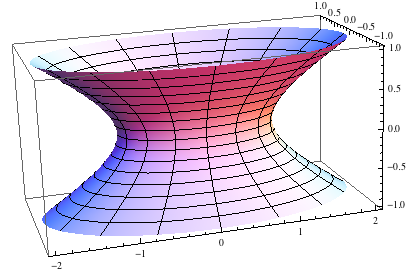}
\caption{A cut off hyperboloid.}\label{fig:hyperevents}
\end{center}
\end{figure}

We will now examine what the persistence diagrams that correspond to the height function in direction $v$ are, where $v$ is a unit vector $(v_1,v_2, v_3)$ with $v_1,v_2,v_3\geq0$. By symmetry we can then deduce the rest of the PHT. Let $M=Hyp(a,b,c)$. Recall that to compute $\PHT(M)(v)$ we care about the filtration of subsets of the form $M(v)_r:=\{ (x,y,z)\in M: (x,y,z)\cdot v\leq r\}$.
Changes in homology as we progress through this subset can occur only when either first contacting or completing one of the boundary ellipses or when the $M(v)_r$ encounters a point in $M$ whose normal is $\pm v$. We will now discuss each of these scenarios and find formulae for at what heights they occur.

\begin{table}
\begin{center}
\begin{tabular}{c|c|c}
&$j=1$ & $j=2$\\
\hline
$r_{l,j}(v,a,b,c)$ & $-\sqrt{1+\frac{1}{c^2}}\sqrt{a^2v_1^2 + b^2v_2^2} -v_3$ & $\sqrt{1+\frac{1}{c^2}}\sqrt{a^2v_1^2 + b^2v_2^2} -v_3$\\
\hline
$r_{u,j}(v,a,b,c) $ &$-\sqrt{1+\frac{1}{c^2}}\sqrt{a^2v_1^2 + b^2v_2^2} +v_3$ & $\sqrt{1+\frac{1}{c^2}}\sqrt{a^2v_1^2 + b^2v_2^2} +v_3$\\
\hline
$r_{i,j}(v,a,b,c)$ & $-\sqrt{a^2v_1^2 + b^2v_2^2-c^2 v_3^2}$ &$\sqrt{a^2v_1^2 + b^2v_2^2-c^2 v_3^2}$\\
(if it occurs)
\end{tabular}
\caption{Table of the heights at which homological changes may happen}\label{table:ellipsoidheights}
\end{center}
\end{table}

Let $r_{l, 1}$ and $r_{l,2}$, and $r_{u, 1}$ and $r_{u,2}$, be the heights that $M(v)_r$ first contacts and then completes the lower and upper boundary ellipses respectively. To compute $r_{l, 1}$ and $r_{l,2}$, or $r_{u, 1}$ and $r_{u,2}$, we can use the method of Langrange multipliers where we wish to find the extreme values of $f_l(x,y) =v_1x+ v_2y -v_3$, or $f_u(x,y) =v_1x+ v_2y +v_3$, with the constraint that $g(x,y)= \frac{x^2}{a^2} + \frac{y^2}{b^2} =1+ \frac{1}{c^2}$. The formulae we calculated are in Table A.1.

For some directions $v$ there will exist points $(x_1, y_1, z_1)$ and $(x_2, y_2, z_2)$ in $M$ with normal $\pm v$. If they exist, denote the heights at which $M(v)_r$ encounters such points by $r_{i,1}$ and $r_{i,2}$. If they exist, such points will satisfy $(\frac{2x}{a^2}, \frac{2y}{a^2}, \frac{2z}{a^2}) = \lambda (v-1, v_2, v_3)$ for some $\lambda$. Using the constraint $\frac{x^2}{a^2} + \frac{y^2}{b^2} - \frac{z^2}{c^2}=1$ we compute these heights (if they exist) and these are shown in Table A.1.
The corresponding points in $M$ (with normal $\pm v$) have $z$ coordinate $\pm c^2v_3/\sqrt{a^2v_1^2 + b^2v_2^2-c^2 v_3^2}$ and hence they exist in the cut off hyperboloid exactly when $c^2v_3/\sqrt{a^2v_1^2 + b^2v_2^2-c^2 v_3^2}< 1$. That is when 
\begin{align}\label{eq:condextra}
c^4v_3^2<a^2v_1^2 + b^2v_2^2-c^2 v_3^2.
\end{align}
%It can be shown algebraically that \eqref{eq:condextra} is equivalent to  $r_{u,1}(v,a,b,c)< r_{i,1}(v,a,b,c)$ (both occurring exactly there are points with normal $\pm v$ in $Hyp(a,b,c)$). 
%
All non-boundary points in $Hyp(a,b,c)$ are saddle points and thus they must be critical points of index $1$  for the Morse function defined as the height function in the direction of $v$. Being an index $1$ critical points means that including it must change the homology  -- causing either a decrease in $H_0$ or an increase in $H_1$ when it is encountered.  

%If there exist points in $M$ with normal $\pm v$, such points will satisfy $(\frac{2x}{a^2}, \frac{2y}{a^2}, \frac{2z}{a^2}) = \lambda (v-1, v_2, v_3)$ for some $\lambda$. 
%The lower boundary is first contacted at
%$$r_{l,1}(v,a,b,c)=-\sqrt{1+\frac{1}{c^2}}\sqrt{a^2v_1^2 + b^2v_2^2} -v_3$$
%The lower boundary is completed at 
%$$r_{l,2}(v,a,b,c)=\sqrt{1+\frac{1}{c^2}}\sqrt{a^2v_1^2 + b^2v_2^2} -v_3.$$
%Similar calculations show that the upper boundary is first contacted at
%$$r_{u,1}(v,a,b,c)=-\sqrt{1+\frac{1}{c^2}}\sqrt{a^2v_1^2 + b^2v_2^2} +v_3$$
%and the lower boundary is completed at 
%$$r_{u,2}(v,a,b,c)=\sqrt{1+\frac{1}{c^2}}\sqrt{a^2v_1^2 + b^2v_2^2} +v_3.$$
%

The $X_0(Hyp(a,b,c),v)$ will always have  one essential class which is at $(r_{l,1}, \infty)$. If there are internal points with normal $\pm v$ then the upper boundary ellipse is contacted before any path from the lower boundary ellipse to the upper boundary ellipse is seen. Thus the upper boundary first appears as a second connected component. It then joins the first connected component once a path between the boundary ellipses is first completed which is when the first of the two point with normal $\pm v$ is included. The two components merge at height $r_{i,1}$. We thus have the point $(r_{u,1},  r_{i,1})$ as a second off diagonal point in $X_0(Hyp(a,b,c),v)$. Figure \ref{F:eccurves} shows the progression of important sublevel sets of the height function in direction $v$ over a cut off hyperboloid in this case.
\begin{figure}[hbt]
\begin{center}
\begin{tabular}{cc}
\begin{tabular}{c}
\includegraphics[width=2in]{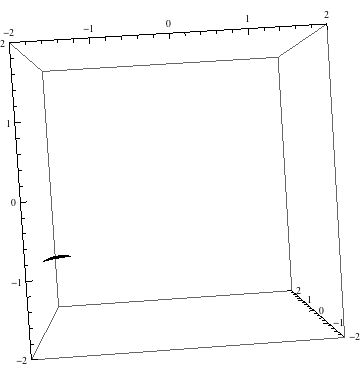}
\end{tabular}
\qquad &  \qquad
\begin{tabular}{c}
\includegraphics[width=2in]{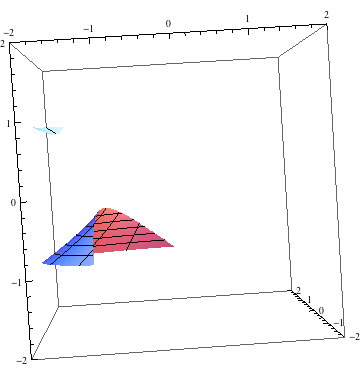}
\end{tabular}
\\
 \qquad (a) \qquad &  \qquad (b)\\
 \begin{tabular}{c}
\includegraphics[width=2in]{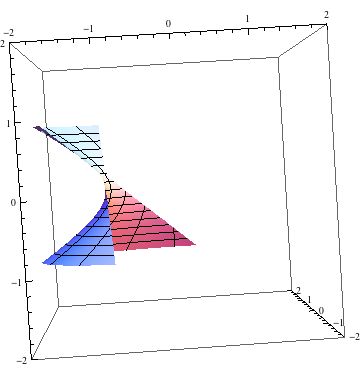}
\end{tabular}
\qquad &  \qquad
\begin{tabular}{c}
\includegraphics[width=2in]{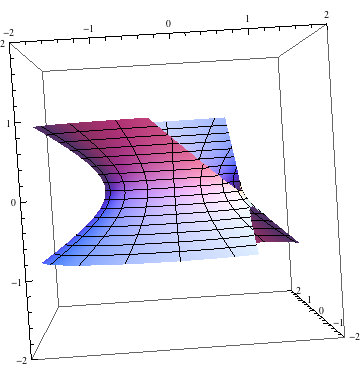}
\end{tabular}
\\
 \qquad (c) \qquad &  \qquad (d)\\
\end{tabular}
\caption{One possible progression of the sublevel sets of the height function in direction $v$ over a cut off hyperboloid. (a) The first time a sublevel set is non-empty an essential $H_0$ class is born. (b) Here the sublevel set touches the upper boundary without containing any path between the upper and lower boundary. A second $H_0$ persistent homology class is born. (c) A path between the upper and lower boundaries is found - happening when a point on the cut off hyperboloid has normal $\pm v$. The second $H_0$ persistent homology class dies. (d) An essential $H_1$ class is born- happening when another point on the cut off hyperboloid has normal $\pm v$. No changes in homology occur after this.}\label{F:eccurves}
\end{center}
\end{figure}

There is always exactly one essential class in $X_1(Hyp(a,b,c),v)$ and no other off diagonal points. The essential class is born when the loop around the hyperboloid is first completed.  If there are internal points with normal $v$ it will occur when the second of these points appears (at height $r_{i,1}$). Otherwise the loop first appears when the lower boundary loop is completed (at height $r_{l,2}$). In conclusion,  if $c^4v_3^2<a^2v_1^2 + b^2v_2^2-c^2 v_3^2$ then
\begin{align*}
X_0(Hyp(a,b,c),v)&=\{(r_{l,1}, \infty), (r_{u,1},  r_{i,1})\}\\
X_1(Hyp(a,b,c),v)&=\{(r_{i,2},\infty)\}
\end{align*}
and if $c^4v_3^2\geq a^2v_1^2 + b^2v_2^2-c^2 v_3^2$ then
\begin{align*}
X_0(Hyp(a,b,c),v)&=\{(r_{l,1}, \infty)\}\\
X_1(Hyp(a,b,c),v)&=\{(r_{l,2},\infty)\}
\end{align*}
where the $r_{l,j},r_{u,j}$ and $r_{i,j}$ are the formulae in Table \ref{table:ellipsoidheights}.

We now will analyze the distances within a one dimensional family of cut-off hyperboloids. We will fix the bounding ellipses to be
$$\{(x,y,z):\frac{x^2}{4} + y^2 = 1\text{ and } |z|=1\}.$$
We still have one parameter of freedom. Algebraically, for each $a\in (0,2)$, we can define a hyperboloid
$Hyp(a)= Hyp(a,a/2,a/\sqrt{4-a^2} )$. These hyperboloids satisfy our desired boundary condition and are determined by where they intersect the $x$ axis; $Hyp(a)$ intersecting at $\pm(a,0,0)$. The advantage of considering this family of cut off hyperboloids is that they have the same convex hull and hence are already are (up to the same constant scaling factor) normalized.

Again we will focus on $v$ in the positive quadrant. We have two different cases for what the diagrams are in direction $v$ for $Hyp(a)$ depending on whether$c^4v_3^2<a^2v_1^2 + b^2v_2^2-c^2 v_3^2$ or not. Given our relationships between $a,b$ and $c$ this condition can be written as
%$$4a^2 v_3^2 -(4-a^2)(a^2 v_1^2 + \frac{a^2}{4} v_2^2)<0$$
%\begin{align}\label{eq:extrapt}
$16v_3^2 <(4-a^2)^2(4 v_1^2 + v_2^2).$
%\end{align}
%If \eqref{eq:extrapt} holds then internal points with normal vector $v$ exist 
%\begin{align*}
%X_0&(Hyp(a),v)\\
%&=\left\{(-\sqrt{4v_1^2+ v_2^2}-v_3, \infty), (-\sqrt{a^2v_1^2 + \frac{a^2}{4}v_2^2 - \frac{a^2}{4-a^2} v_3^2}, -\sqrt{4v_1^2+ v_2^2}+v_3),  \right\}\\X_1&(Hyp(a),v)=\left\{(\sqrt{a^2v_1^2 + \frac{a^2}{4}v_2^2 - \frac{a^2}{4-a^2} v_3^2},\infty) \right\}.
%\end{align*}
%If  \eqref{eq:extrapt} does not hold then 
%\begin{align*}
%X_0(Hyp(a),v)&=\left\{(-\sqrt{4v_1^2+ v_2^2}-v_3, \infty) \right\}\\
%X_1(Hyp(a),v)&=\left\{(\sqrt{4v_1^2+ v_2^2}-v_3, \infty)\right\}.
%\end{align*}
% exist and the $H_0$ diagram is 
%$$\left\{(-\sqrt{4v_1^2+ v_2^2}-v_3, \infty), (-\sqrt{a^2v_1^2 + \frac{a^2}{4}v_2^2 - \frac{a^2}{4-a^2} v_3^2}, -\sqrt{4v_1^2+ v_2^2}+v_3),  \right\}$$ and the $H_1$ diagram is 
%$$\left\{(\sqrt{a^2v_1^2 + \frac{a^2}{4}v_2^2 - \frac{a^2}{4-a^2} v_3^2},\infty) \right\}.$$
%If  \eqref{eq:extrapt} does not hold then the $H_0$ diagram is 
%$$\left\{(-\sqrt{4v_1^2+ v_2^2}-v_3, \infty) \right\}$$ and the $H_1$ diagram is 
%$$\left\{(\sqrt{4v_1^2+ v_2^2}-v_3, \infty)\right\}.$$
%
Importantly the formulae $r_{u,j}$ and $r_{l,j}$ are independent of  $a$ because they only depend on what height the boundary ellipses are contacted or completed, and these boundary ellipses are independent of $a$.

The $X_1(Hyp(a),v)$ only contains one essential class regardless of the direction $v$. Thus to compute the distance between $X_1(Hyp(a_1),v)$ and $X_1(Hyp(a_2),v)$ we just take the distance between the first coordinate of the only point in each diagram. The essential classes in the $H_0$ persistent homology, for every fixed $v$, are the same for all of the $Hyp(a)$, regardless of $a$, and so when computing the distance between $X_0(Hyp(a_1),v)$ and $X_0(Hyp(a_2),v)$ we can effectively ignore them. If both diagrams contain a finite persistence off diagonal point then they will have the same birth times (for the fixed direction $v$) and hence we should match them to each other rather than both to the diagonal.\footnote{This is an advantage of our choice of $p=1$ in the distance function on the space of persistence diagrams. If we were to use a different metric on the space of persistence diagrams this conclusion (that the two off diagonal point would be paired because they have the same birth time) would not generally hold.} If only one has a finite persistence off diagonal point then it has to be matched to the diagonal. 

We computed the distances (with small error due to using a finite approximation) via our algorithm. By inspection of the diagrams in the various cases we can see that the distance between the $\PHT$s is in fact exactly double the distance between the  $\PHzeroT$s. This would not hold for general cut off hyperboloids as it stems from having those fixed boundary ellipses. Since the $\PHzeroT$s are significantly faster to compute we computed these instead.

We considered the set of shapes $\{Hyp(a)\}$ for $a$ in $0.125$ increments from $0.125$ to $1.875$. 
%$a\in \{0.125,0.25,0.375,0.5,0.625, 0.75,0.875, 1, 1.125,1.25,1.375,1.5,1.625, 1.75,1.875\}$. 
After computing the matrix of distances we performed MDS. There was only one non-zero eigenvalue. In Figure \ref{fig:mdshyper} we plotted the scores in this coordinate with respect to the variable $a$.

\begin{figure}[hbt]
\begin{center}
\includegraphics[height=3in]{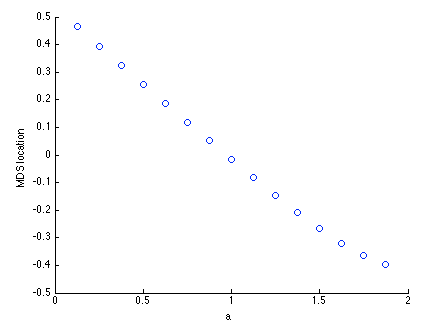}
\caption{The MDS coordinate of $Hyp(a)$}\label{fig:mdshyper}
\end{center}
\end{figure}

\bibliographystyle{plain}
\bibliography{refs}

\end{document}